\documentclass[12pt,a4paper]{amsart}
\usepackage{amsfonts,amssymb}
\usepackage{amsmath}
\usepackage{latexsym}  
\usepackage[left=2cm,right=2cm]{geometry}
\usepackage{todonotes}
\usepackage[utf8]{inputenc}
\usepackage[T1]{fontenc}
\usepackage{hyperref}
\usepackage{romanbar}
\usepackage{enumitem}

\title{Singular semilinear elliptic equations in nondivergence form}
\author{Agnieszka Ka{\l}amajska}
\address{Faculty of Mathematics, Informatics and Mechanics, University of Warsaw, Banacha 2, 02-097 Warsaw, Poland}
\email{A.Kalamajska@mimuw.edu.pl}
\urladdr{0000-0001-5674-8059}

\author{Dalimil Pe\v{s}a}
\address{Dalimil Pe\v{s}a, Department of Automation and Mathematics, Faculty of Electrical Engineering and Informatics, University of Pardubice, Studentská 95, 532 10 Pardubice, Czech Republic}
\email{dalimil.pesa@upce.cz}
\urladdr{0000-0001-6638-0913}

\author{Artur Rutkowski}
\address{Wroc\l{}aw University of Science and Technology, Wyb. Wyspia\'{n}skiego 27, 50-370 Wroc\l{}aw, Poland}
\email{artur.rutkowski@pwr.edu.pl}
\urladdr{0000-0002-6466-2105}

\usepackage{color}

\newtheorem{theorem}{\bf Theorem}[section]

\newtheorem{lemma}[theorem]{\bf Lemma}

\newtheorem{o-problem}[theorem]{\bf Open Problem}

\theoremstyle{definition}
\newtheorem*{oq}{\bf Open problems}
\newtheorem{remark}[theorem]{\bf Remark}
\newtheorem{definition}[theorem]{\bf Definition}

\makeatletter \@addtoreset{equation}{section}

\makeatother

\setlist[enumerate,1]{label={\textup{\textbf{\roman*)}}}}

\newcommand{\beq}{\begin{equation}}
\newcommand{\eeq}{\end{equation}}

\newcommand{\R}{\mathbb{R}}

\newcommand{\supp}{\operatorname{supp }}

\newcommand{\sgn}{{\rm sgn}\,}

\newcommand{\A}{\mathbf{A}}

\newcommand*\Div{\operatorname{div}}

\makeatletter
\renewcommand{\subsubsection}{\@startsection{subsubsection}{3}%
    \z@{.5\linespacing\@plus.7\linespacing}{-.5em}%
    {\normalfont\bfseries}}
\makeatother

\allowdisplaybreaks

\begin{document}
\linespread{1.3}

\maketitle

\begin{flushleft}

\end{flushleft}
\begin{abstract}
   We study the singular semilinear equation $-Pu = \frac{f}{u^\gamma}$ on a bounded domain $\Omega$ with Dirichlet condition $u \equiv 0$ on $\partial \Omega$ , where $P$ is a second-order elliptic differential operator in nondivergence form. We obtain the existence of a solution under the assumptions that $\Omega \in C^{1,1}$ and $P$ has $C^1$ coefficients, as well as the uniqueness of solutions in $L^1(\Omega)$, under the assumptions that $\Omega \in C^2$ and $P$ has $C^2$ coefficients. Our proofs are based on a novel combination of tools, such as recently obtained nonlinear variants of Gagliardo--Nirenberg inequalities, estimates of Green functions, and new variants of Kato-type inequalities.
\end{abstract}
\section{Introduction}
In this paper
we are interested in the following Dirichlet problem for a singular semilinear equation with Dirichlet condition
\begin{align}\label{pe}
   \begin{cases} -Pu = \displaystyle\frac{f}{u^\gamma}\quad &\text{in}\ \Omega,\\
   u \equiv 0\quad &\text{on}\ \partial \Omega, 
   \end{cases}
\end{align}
where $\Omega\subseteq\R^d$  is a bounded domain,  $d\geq 2$, 
$f\in L^1(\Omega)$ is nonnegative almost everywhere,  $\gamma>0$, $u$
is nonnegative almost everywhere, and $Pu$ is
an extension of
a second-order uniformly elliptic operator in nondivergence form:
\begin{equation}\label{eq:P}\tag{$P$}
Pu(x) := \sum_{i,j}a_{ij}(x)\frac{\partial^2 u(x)}{\partial x_i\partial x_j}, \ \ u\in W^{2,1}_{loc}(\Omega)
\end{equation}
with the symmetric matrix $\A(x) = [a_{ij}(x)]_{1\leq i,j\leq d}$, $x\in \overline{\Omega}$, satisfying 
\begin{align} \tag{$\A$} \label{ellipticity}
 c_{\A}\|\xi\|^2 \le \A(x)\xi \cdot \xi \le C_{\A} \|\xi\|^2,\quad x\in \overline{\Omega},\ \xi\in \R^d \quad \text{ for some } 0< c_{\A} \leq C_{\A} <\infty.
\end{align}
We first give a brief explanation of the notions of solution used in the paper. The complete definitions are given in Section~\ref{prelimin} below.  
The boundary condition is interpreted as 
 \begin{align}\tag{Tr}
 \label{eq:Dirichletcond}
        \lim\limits_{\epsilon\to 0}\frac 1 \epsilon \int_{\Omega_\epsilon} |u| = 0, \ {\rm where} \ \ \Omega_\epsilon = \{x\in \Omega: \delta_\Omega(x) < \epsilon\}.
\end{align}

 When $u\in W^{2,1}_{loc}(\Omega)$, the operator $Pu$ can be computed almost everywhere in $\Omega$ in the classical sense, but in general we expect the solutions to be less regular.

If $\A\in C^2(\Omega)$ and $u\in L^1_{loc}(\Omega)$, then we can extend the operator to the distributional setting:
\begin{align}
  \langle -Pu,\phi    \rangle &:= 
-\int_\Omega u P^*\phi,\ {\rm where} \ \ \phi \in C^\infty_c(\Omega),\label{P-distr}\\
P^*\phi&:=  \Div\left( \A\nabla \phi \right) + \Div\left(\Div(\A)\phi  \right) \hbox{\rm --- the adjoint operator}\label{P-star}
\end{align}
and $\Div\A:= (\Div \A^1,\dots , \Div\A^d)^T$, where $\operatorname{div} \mathbf{A}^i$ is the standard divergence operator applied to the column $\A^i$  of the matrix $\A$. 
Note that $P^*\phi\in C(\overline{\Omega})$. 

With more precise information about $u$, that is, when $u\in W^{1,1}_{loc}(\Omega)$, we can verify that 
\begin{align}\label{dist-special}
  \langle -Pu,\phi    \rangle &= 
\int_\Omega \A\nabla u \cdot \nabla \phi + (\Div\A \cdot \nabla u )\phi , \ \ \phi\in C_c^\infty (\Omega).  
\end{align}
Here, it suffices to assume that $\A\in C^1(\Omega)$.

\smallskip
The distributional formulations \eqref{P-distr} and \eqref{dist-special} lead to two slightly different definitions of solutions: we show existence in the framework of \eqref{dist-special} in Theorem~\ref{maintheorem}, which requires only $\A \in C^1(\overline{\Omega})$, and uniqueness for \eqref{P-distr} in Theorem~\ref{theoremUniqness}, which imposes more regularity on $\A$, specifically $\A \in C^2(\overline{\Omega})$, but allows less regular solutions $u$. 

We also require that $f/u^\gamma$ be locally integrable, so that the right-hand side of \eqref{pe} is a distribution given by a locally integrable function. Then the equality $-Pu=f/u^\gamma$ can be interpreted as the equality of two distributions.

\smallskip

Studies of singular semilinear equations of the form similar to \eqref{pe} were initiated over half a century ago with the works of Fulks and Maybee \cite{MR123095}, Stuart \cite{MR404854}, and Crandall, Rabinowitz, and Tartar \cite{MR427826}. Since then, the topic has been extensively investigated to cover broader cases of $P$, $f$, and $\gamma$, as well as various generalizations or perturbations of \eqref{pe}. We refer to the recent paper \cite{MR4860322} for an exposition of the existing results. 

Our problem can be considered as a multidimensional counterpart of the Emden--Fowler equation, originally an ODE of the form $u^{''}(x)=Cx^\beta u^\alpha $, where $\alpha,\beta\in \R$. 
This equation and its prominent generalizations, Lane--Emden--Fowler 
type equations, have been applied in science, electrical engineering, astrophysics, and the analysis of non-Newtonian fluids \cite{Emden-Fowleraplication,cosmologicalModel, LUNING1984145, Emden-FowlerOverview}. 

More recently, Boccardo and Orsina \cite{MR2592976} made a substantial contribution to this class of equations by investigating the existence of solutions to \eqref{pe} 
for operators $P$ in divergence form, i.e.~when $Pu=\Div\left(\A\nabla u\right)$, and with nonnegative $f\in L^1(\Omega)$.

Elliptic operators in nondivergence form are of considerable interest in the theory of stochastic processes, as they form an important subclass of generators of diffusion processes, see e.g.~\cite[Section~6.7.3]{MR2072890}. Stochastic analysis largely facilitates the construction and estimates of integral kernels, e.g.~the Green function or the Poisson kernel; we refer to \cite{MR4554661, MR842803} in this context.

\smallskip
Our first main result is the following theorem on the existence of solutions to \eqref{pe} with the same range of $f$ and $\gamma$ as in \cite{MR2592976}, but for $P$ in nondivergence form. For the symbols and notions used in the formulation, see Section~\ref{prelimin}.

\begin{theorem}[existence and regularity]\label{maintheorem} Assume that $\Omega \in C^{1,1}$, $\A\in C^{1}(\overline{\Omega})$ satisfies \eqref{ellipticity}, and 
\begin{align}\label{eq:smallA}
\left( \frac{\gamma + 1}{ \gamma} \right)^2 \lVert \Div \A \rVert_{L^\infty(\Omega)}c_{\A}^{-1} C_{P,\Omega} <1.
\end{align}
Then, for any nonnegative and nontrivial $f\in L^1(\Omega)$ and $\gamma > 0$
there exists a weak solution $u\in W^{1,1}_{loc}(\Omega)\cap L^1(\Omega)$ to \eqref{pe}, which is positive and locally bounded from below in $\Omega$ by positive constants. Furthermore, $u$ has the following regularity:
\begin{enumerate}[label={\rm \bf (\alph*)}]
    \item \label{item:gammageq2} $u\in H^1_{loc}(\Omega)$ for $\gamma\geq 2$,
    \item \label{item:gammain12} $u\in H^1_{loc}(\Omega) \cap W_{ 0}^{1,q}(\Omega)$ with any $1\le q < \frac{2}{\gamma}$, for $\gamma \in (1,2)$,
    \item \label{item:gamma1} $u\in H^1_0(\Omega)$ for $\gamma=1$,
    \item \label{item:gammaleq1} $u\in W^{1,q}_0(\Omega)$ with $q=d(1+\gamma)/(d-(1-\gamma))$, for $\gamma\in (0,1)$.
\end{enumerate}
In addition, for all $\gamma\in (0,\infty)$ we have
$u^{\frac{\gamma +1}{2}}\in H^1_0(\Omega)$ and
 the right-hand side of \eqref{pe} is locally integrable.  Finally, we have $u\gtrsim \delta_\Omega$.
\end{theorem}

   The inequality \eqref{eq:smallA} can be interpreted as a smallness condition for $\Div \A$.\smallskip
    
    The proof of Theorem~\ref{maintheorem} is given in Section~\ref{sec:existenceproof}.  The general strategy follows the classical works \cite{MR2592976,MR427826} and consists in constructing the solution by approximating it with the solutions to regularized problems. The approximate solutions are obtained by the Schauder Fixed Point Theorem. Then, the key points are to prove that the approximating sequence is monotone and enjoys suitable a~priori Sobolev-type estimates; the approach to these two issues strongly depends on the exact setting of the problem.

\smallskip
Despite a large number of papers that followed \cite{MR2592976}, it seems that \eqref{pe} with nondivergence form diffusions under present assumptions on $f$ and $\gamma$ was not yet investigated. We note however, that since $\A\in C^1(\overline{\Omega})$, the weak definition \eqref{dist-special} of $P$ is similar to
that considered for divergence form equations with a bounded linear drift term in \cite{MR3273407}, but only $f\in L^\infty(\Omega)$ are considered there. We also refer to \cite{MR4603892,MR4611538,MR3927527} for singular semilinear problems with additional nonlinear gradient terms.

Although the formulation of the problem seems almost the same as in
\cite{MR2592976}, our approach with the nondivergence form operator causes technical differences in the proofs and requires some new methods to carry out the approximation procedure used for existence. Here are the main three points which are different compared to the proof in the divergence form case in \cite{MR2592976}:
\begin{itemize}
    \item A crucial new tool we used to prove the existence result are the recently obtained a~priori estimates related to the nonlinear Gagliardo--Nirenberg inequality:
        $$
                \int_\Omega \|\nabla u\|_\A^2 u^{\gamma -1} \le C \int_\Omega \lvert Pu \rvert u^\gamma, 
        $$
        deduced from  \cite{MR4948792}, see Lemma \ref{lemm:KPR_estimate} below.
    \item  We use the representation of solutions with the Green function $G_\Omega$ together with the estimates of $G_\Omega$, to obtain sharp decay rates of the approximate solutions. These rates are essential for proving parts \ref{item:gammain12} and \ref{item:gammaleq1} of Theorem~\ref{maintheorem}.
    \item We apply a Kato-type inequality for strong solutions of equations in nondivergence form,  given in Lemma \ref{lem:Katoidentity}, to prove that the approximate solutions form an increasing sequence. 
\end{itemize}

 Let us mention that the regularity results obtained  in Theorem \ref{maintheorem} \ref{item:gammageq2}, \ref{item:gamma1}, \ref{item:gammaleq1}, and the fact that the composition $u^{\frac{\gamma +1}{2}}$ belongs to $H^1_0(\Omega)$
are consistent with the case of the divergence form operator in \cite{MR2592976}
(see \cite{MR2592976}: Theorem 5.6 with $m=1$ for $\gamma \in (0,1)$ and Theorem 3.2 
for $\gamma =1$, as well as the summary in \cite[page~27]{MR4860322}). However, an analogue of the global Sobolev regularity for $\gamma\in (1,2)$ in \ref{item:gammain12} is apparently not known in the setting of divergence-form operators.
\ \\

\noindent
Our second main result is about the uniqueness of distributional solutions to \eqref{pe}.

\begin{theorem}[uniqueness]\label{uniqregu} \label{theoremUniqness}  Assume that $\Omega \in C^{2}$ and that $\A\in C^{2}(\overline{\Omega})$ satisfies \eqref{ellipticity}. Let $f\in L^1(\Omega)$ be nonnegative and not constantly equal to zero and let $\gamma\in (0,\infty)$. Then, the problem \eqref{pe} has at most one distributional solution $u\in L^1(\Omega)$ (see Definition~\ref{def:distributionalsolution}).
\end{theorem}
The proof of Theorem~\ref{uniqregu} is given in Section~\ref{sec:uniquenessproof}. 
Note that the above result is given in the framework of distributional solutions, which a~priori form a larger class than the weak solutions discussed in Theorem~\ref{maintheorem}. 
Now we are using $P$ as in \eqref{P-distr}, while in Theorem \ref{maintheorem} we understand $P$ as in \eqref{dist-special}.
See also Remark~\ref{rem:weakdist} below for a more precise comment.

\smallskip
For operators in divergence form, a similar uniqueness result was given by Oliva and Petitta~\cite{OLIVA}. We note however, that the class of solutions that we consider, $u\in L^1(\Omega)$, is wider than the one used in \cite{OLIVA}, where $u\in W^{1,1}_{loc}(\Omega)\cap L^1(\Omega)$ is assumed. The structure of the proof is similar to that in \cite{OLIVA}, where the authors relied strongly on a Kato-type inequality, but we need to overcome certain obstacles resulting from considering the operator in nondivergence form and a weaker type of solutions. We outline the proof and the differences below. 
\begin{enumerate}[label={\textup{\textbf{\arabic*)}}}]
    \item The first major step of the proof is to show that $f/u^\gamma$ is integrable with respect to the weight~$\delta_\Omega$, which is done in Lemma~\ref{lem:FL1loc}. This consists of several elements:
    \begin{itemize}[leftmargin=3.9mm]
    \item
    In \cite{OLIVA}, for $u\in W^{1,1}_{loc}(\Omega)\cap L^1(\Omega)$, one can immediately use all $\phi\in C^1_c(\Omega)$ as test functions. In our case, we first need to extend the class of the test functions from $C_c^\infty(\Omega)$ to all $\phi\in C^1_c(\Omega)$ such that $P^*\phi\in L^\infty(\Omega)$. To this end we use a Friedrichs-type lemma about mollifiers, see Section~\ref{sec:firstextension} below.
    \item The weighted integrability is obtained with the use of a special test function $\xi$ solving $-P^*\xi = 1$. We need the bound $\xi\gtrsim \delta_\Omega$ (closely related to Hopf lemma, see Remark~\ref{hopf-lem}) for this function, which is proved in Lemma~\ref{lem:linear}\ref{lem:linear_iic}.
     To achieve this, we use the Green representation of solutions to $-P^*v = F\in L^\infty(\Omega)$ and the sharp estimates of the adjoint Green function $G^*_\Omega(x,y)=G_\Omega(x,y)$.  
     \item The last auxiliary result for the weighted integrability is Lemma~\ref{lem:specialxi}, where we show that $\xi$ can be approximated from below by functions $\xi_n \in C^1_c(\Omega)$ such that $-P^*\xi_n$ are uniformly bounded from above. The adjoint operator $P^*$ has an essentially different structure than $P$, see also \eqref{eq:P*non-div-P} for other representations, which makes the calculations here more involved than in the divergence form case.
     \end{itemize}
    \item After establishing the weighted integrability, we further extend the class of test functions to the set $C^{1,P}_0(\overline{\Omega})$ consisting of $\phi\in C^1(\overline{\Omega})$ such that $\phi\equiv 0$ on $\partial \Omega$ and $P^*\phi\in L^\infty(\Omega)$, thus showing that $u$ is a very weak solution in the sense of Definition~\ref{very-weak}. Let us emphasize that up to this point we still only assume that $u\in L^1(\Omega)$, which makes this result interesting by itself.
    \item The last step is to prove the Kato-type inequality for very weak solutions of equations in nondivergence form (Lemma~\ref{lem:Kato}) and to apply it to deduce uniqueness.
\end{enumerate}

\smallskip

 The paper is organized as follows. Section~\ref{prelimin} contains the preliminary material. In Section~\ref{sec:existenceproof} we prove the existence of solutions (Theorem~\ref{maintheorem}) together with the necessary auxiliary results. Section~\ref{sec:uniquenessproof} is devoted to the proof of Theorem~\ref{uniqregu}, the uniqueness result. Finally, in Section~\ref{sec:discussion} we discuss possible extensions and open problems. 

\section{Preliminaries}\label{prelimin}
\subsection{General notation}\label{general-notation} \ \medskip \\
\noindent {\bf Geometry.} Recall that $\Omega$ is a bounded domain in $\R^d$, $d\geq 2$. By saying that $\Omega\in C^2$ (resp. $C^{0,1}, C^{1,1}$), we mean that there exists a finite number of balls $B_1,\ldots,B_k$ centered at points of $\partial \Omega$ and covering it, such that up to rigid motions $\partial\Omega\cap B_i$ are graphs of $C^2$ (resp. $C^{0,1}$---Lipschitz, $C^{1,1}$---continuously differentiable with Lipschitz gradient) functions $f_i$, $\Omega\cap B_i \subset \{y > f_i\}$, and ${\rm Int}(\Omega^c)\cap B_i \subset \{y < f_i\}$. See e.g.~\cite{MR482102} for further information on $C^{0,1}$ domains and \cite{MR2286038} for the characterization of $C^{1,1}$ domains by the exterior and interior ball conditions. We also define 
\begin{equation}\label{d-omega}
\delta_\Omega(x) := \inf\{\lVert x-y \rVert: y\in \partial \Omega\}.
\end{equation}\smallskip

\noindent {\bf Function spaces.} By $W^{k,p}(\Omega,\mathbb{R}^d)$ we denote the standard Sobolev space of mappings from $\Omega$ to the Euclidean space $\mathbb{R}^d$. If it is clear from the context, we omit $\R^d$ in the notation. We also denote $H^k(\Omega) := W^{k,2}(\Omega)$.
By using the subscript $loc$ we mean the local version of the given space. The class of smooth compactly supported functions in $\Omega$ is denoted by $C_c^\infty(\Omega)$. The spaces $W^{k,p}_{0}(\Omega)$ and $H_0^1(\Omega)$ are the closures of  $C_c^\infty(\Omega)$ in $W^{k,p}(\Omega)$ and $H^1(\Omega)$ respectively. By $X^*$ we mean the dual space to space $X$.\medskip

\noindent {\bf Matrices and vectors.} We work with vectors only in columns. We denote the scalar product of vectors or matrices $v$ and $w$ by $v\cdot w$. The divergence of the matrix field $\A=(\A^1,\dots , \A^d)\in \mathbf{R}^{n\times n}$ (where $A^i\in \mathbf{R}^d$) is the vector field $\operatorname{div}\A:= (\Div\A^1,\dots , \Div\A^d)^T$, where $\operatorname{div} \mathbf{A}^i$ is the standard divergence of column $\A^i$ of $\A$.

By constant $C$ in an inequality we mean that there is a positive constant, whose value does not impact the statement and it can be different even in the same line. 

In some occurrences, we specify the constant marking it by index and explaining further. The notation $X\lesssim Y$ means that $X\leq C Y$ for terms $X,Y$ in question, where $C$ does not depend on $X$ and $Y$; by $X\approx Y$ we mean that both $X\lesssim Y$ and $Y\lesssim X$ hold. \medskip

\noindent \textbf{Mollifier.} Let $\varphi\in C_c^\infty(\R^d)$ be the standard mollifier, i.e.~a nonnegative radial function with support in $B(0,1)$ such that $\int \varphi = 1$. We denote $\varphi_\epsilon(x) = \epsilon^{-d}\varphi(x/\epsilon)$, $\epsilon > 0$.
\subsection{Specific notation} \ \medskip \\ 
\textbf{The norm induced by the matrix $\A$.} If $\A$ satisfies \eqref{ellipticity}, then 
\begin{align*}
    \|\xi\|_{\A(x)} := \sqrt{ \A(x)\xi\cdot \xi},\quad \xi\in \R^d
\end{align*}
is a norm uniformly equivalent to the Euclidean norm, that is, $\|\xi\|\approx\| \xi\|_{\A(x)}$ uniformly in $x\in \overline{\Omega}$ and $\xi\in \R^d$. \smallskip

\noindent\textbf{Poincar\'e inequality.} For a bounded domain $\Omega$ let $C_{P,\Omega}$ be the smallest positive number such that for all $u\in W^{1,1}_0(\Omega)$
\begin{align}
    \int_\Omega |u|&\leq C_{P,\Omega}\int_\Omega \|\nabla u\|. \label{eq:Poincare_1}
\end{align}

\subsection{Notions of solution} \

\smallskip
\noindent

We present two definitions of solutions of \eqref{pe}.

\begin{definition}[solutions of \eqref{pe}]
Let us consider \eqref{pe}, where $f\in L^1(\Omega)$.
\begin{description}
\item[i) weak solution]\label{def:weaksolution}
Assume that $\A\in C^1(\overline{\Omega})$.
 We say that the function
$u\in W^{1,1}_{loc}(\Omega)$ 
is a weak solution to \eqref{pe}, if $u>0$ a.e. in $\Omega$, $u$ satisfies \eqref{eq:Dirichletcond}, $f/u^\gamma \in L^1_{loc}(\Omega)$, and
the two distributions: $-Pu$ defined by \eqref{dist-special}  and $f/u^\gamma$ are equal in $D^{'}(\Omega)$, that is
\begin{align*}
    \int_\Omega \A\nabla u \cdot \nabla \phi + (\Div\A \cdot \nabla u )\phi  = \int_\Omega \frac{f\phi}{u^\gamma}\, dx, \quad  \phi\in C^\infty_c(\Omega).
\end{align*} 
\item[ii) distributional solution]\label{def:distributionalsolution}
Assume that $\A\in C^2(\overline{\Omega})$.
    We say that $u\in L^1(\Omega)$ is a distributional solution to \eqref{pe}, 
    if $u>0$ a.e. in $\Omega$, $u$ satisfies \eqref{eq:Dirichletcond}, $f/u^\gamma \in L^1_{loc}(\Omega)$, and the two distributions: $-Pu$ defined by \eqref{P-distr}  and $f/u^\gamma$ are equal in $D^{'}(\Omega)$, that is
    \begin{align}\label{eq:distributional}
        -\int_\Omega uP^*\phi = \int_\Omega \frac{f\phi}{u^\gamma},\quad  \phi\in C^\infty_c(\Omega).
    \end{align}
\end{description}
\end{definition}
\noindent The above definition also applies to the case of $\gamma=0$ in \eqref{pe}.

\begin{remark}\label{rem:weakdist}
Any weak solution to \eqref{pe} is also a distributional solution, provided that $\A\in C^2(\overline{\Omega})$. The other implication only holds if one can show that the distributional solution is in $W^{1,1}_{loc}(\Omega)$. While there are results giving $W^{1,1}_{loc}(\Omega)$ regularity for functions satisfying the integral identity in \eqref{eq:distributional}, see e.g.~\cite[Corollary~2.9]{MR2103694}, they require a substantially larger class of admissible test functions. In our uniqueness proof, under the assumptions of Theorem~\ref{theoremUniqness}, we show that such extension of the class of test functions is indeed possible, see Lemma~\ref{lem:veryweak}, thus showing that the distributional solutions defined above are in $W^{1,1}_{loc}(\Omega)$. 
\end{remark}

We will also consider strong solutions of linear equations, when the PDE can be interpreted as holding almost everywhere in the domain $\Omega$.

\begin{definition}{\rm (strong solutions of linear equations)}\label{defStrongSol} Let $F$ be a measurable function. We say that $u\in W^{2,1}_{loc}(\Omega)$ is a strong solution to 
\begin{align}\label{eq:linearstrong}
\begin{cases}
    -Pu = F,\quad &\textnormal{in}\ \Omega,\\
    u\equiv0,\quad &\textnormal{on}\ \partial\Omega,
\end{cases}
\end{align}
    if the equation $-Pu(x) = F(x)$ holds almost everywhere in $\Omega$ via \eqref{eq:P} and $u \in W^{1,p}_0(\Omega)$ for some $p\in (1,\infty)$.
    
\end{definition}
\begin{remark}
    Note that for $u\in W^{2,1}_{loc}(\Omega)$, by the Nikodym ACL characterization \cite[Section~1.1.3]{MR2777530}, $\nabla u$ and  $\nabla^{(2)}u$, so also $Pu$, are classically well-defined almost everywhere in $\Omega$.
\end{remark}

\section{Proof of Theorem \ref{maintheorem}. The existence result}\label{sec:existenceproof}

\subsection{Four auxiliary lemmas}\label{sec:lemmas}\

 Given  $w\in L^1(\Omega)$, we can define it at every point $x\in \overline{\Omega}$, by choosing its canonical Borel representative, that is, 
 by the formula
\begin{equation}\label{valueatpoint}
w(x):=\limsup_{r\to 0^+} \frac{1}{|\Omega\cap B(x,r)|}\int_{\Omega\cap B(x,r)} w(y)\, dy \in [-\infty,+\infty], \ x\in \overline{\Omega}.
\end{equation}
It is known that when $\Omega$ is a bounded Lipschitz domain, such a representative coincides with any other ${\mathcal L}^d$-almost everywhere in $\Omega$. Moreover, if $u\in W^{1,1}(\Omega)$, then it coincides $\sigma$-almost everywhere on $\partial\Omega$ with the classical trace of $w$, see e.g.~\cite{MR555952}. In particular, the formula \eqref{valueatpoint} can be used to define the trace of $w$ on $\partial\Omega$.\\

One of the essential tools in our estimates is the following lemma based on the results from the recent paper \cite{MR4948792}. In this lemma we assume that $u$ is defined at every point $x\in \overline{\Omega}$ by the formula \eqref{valueatpoint}.

\begin{lemma}[nonlinear inequality for $Pu$]\label{lemm:KPR_estimate}
    Let $\Omega\in C^{0,1}$, $\gamma>0$, and let $\A\in C^1(\overline{\Omega})$ satisfy \eqref{ellipticity}. Let further the operator $P$ be given by \eqref{eq:P} and assume
    \begin{align}\label{eq:KPR_assumption}
        \left( \frac{\gamma + 1}{ \gamma} \right)^2 \lVert \operatorname{div} \A \rVert_{L^\infty(\Omega)}c_{\A}^{-1} C_{P,\Omega}^2 <1, 
    \end{align}
    where 
    $c_{\A}$ is the lower ellipticity constant from \eqref{ellipticity} and $C_{P,\Omega}$ is the optimal constant for the Poincar\'e inequality in $\Omega$ from \eqref{eq:Poincare_1}. 
    
    Then for every $u \in W_{\rm loc}^{2, 1}(\Omega)$ satisfying $u \equiv 0$ on $\partial \Omega$, $u>0$  inside $\Omega$, $\lvert Pu \rvert u^{\gamma} \in L^1(\Omega)$, and $u^{\gamma + 1} \in W^{2, 1}(\Omega)$ we have
    \begin{equation*}
        \int_\Omega \|\nabla u\|_\A^2 u^{\gamma -1} \le C \int_\Omega \lvert Pu \rvert u^\gamma ,
    \end{equation*}
    where the constant $C$ does not depend on $u$.
\end{lemma}

\begin{proof} 
First, let us explain the strategy. The estimate will follow from \cite[Theorems~4.1 and 4.4]{MR4948792} once we verify the relevant assumptions. Specifically, \cite[Theorem~4.4, ii)]{MR4948792} implies
    \begin{equation*}
        \int_\Omega \|\nabla u\|_{\A}^2 u^{\gamma -1} \lesssim \int_\Omega \lvert Pu \rvert u^\gamma + \Theta,
    \end{equation*}
    where $\Theta$ is a certain boundary term. Then \cite[Theorem~4.1, i)]{MR4948792} implies $\Theta \leq 0$, so this term can be omitted in the estimates.
    
    We shall now verify the assumptions of the required theorems. We do not repeat their full formulation because of its substantial length and the fact that we only require a specific case of the more general inequalities. We only explain how individual assumptions are verified, staying consistent with the notation of \cite{MR4948792} and providing the relevant page numbers. We note that the paper \cite{MR4948792} deals precisely with operators in nondivergence form as in \eqref{eq:P}. We consider the following nonlinearities 
     \begin{align*}
        h(s) = \gamma s^{\gamma-1}, \quad 
        H(s) = s^{\gamma}, \quad 
        \widetilde{H}(s) = \frac{1}{\gamma+1} s^{\gamma+1},\quad  \text{for } s \in (0, \infty).
    \end{align*}
    Then $\widetilde{H}' = H$ and $H' = h$ on $(0, \infty)$, while $H$ and $\widetilde{H}$ clearly have continuous extension to $[0, \infty)$. Hence, our triplet of functions $(h, H, \widetilde{H})$ satisfies the assumption \textbf{(I)} (\cite[p.~8]{MR4948792}) for the interval $I = [0, \infty)$. The general assumption \textbf{(G)} (\cite[p.~7]{MR4948792}) is satisfied as $h$ is positive, $\Omega\in C^{1,1}$, $\A\in C^1(\overline{\Omega})$, and \eqref{ellipticity} holds. The assumption \textbf{(u-I)} (\cite[p.~8]{MR4948792}) is satisfied because of our assumptions on $u$. Furthermore, $\widetilde{H} \geq 0$, $\widetilde{H}(0) = 0$, so the assumption $u \equiv 0$ on $\partial \Omega$ ensures $\widetilde{H}(u) \equiv 0$ on $\partial \Omega$ and so the assumption \textbf{(u-$\widetilde{\text{H}}$)} (\cite[p.~9]{MR4948792}) is satisfied as well. As
    \begin{equation*}
        \frac{H^2(s)}{h(s)} = \frac{1}{\gamma} s^{\gamma+1} = \frac{\gamma+1}{\gamma} \widetilde{H}(s),
    \end{equation*}
    we find that the assumption {\bf ($\mathcal{G}_H$)} (\cite[p.~9]{MR4948792}) also holds, with the optimal value of the constant $C_{\widetilde{H}}$ being $\frac{\gamma+1}{\gamma}$.

    Hence, all the assumptions of \cite[Theorem~4.1, i)]{MR4948792} (\textbf{(G)}, \textbf{(I)}, \textbf{(u-I)}, \textbf{(u-$\widetilde{\text{H}}$)}) have been verified. Regarding \cite[Theorem~4.4, ii)]{MR4948792}, we have verified the assumptions \textbf{(G)}, \textbf{(I)}, \textbf{(u-I)}, and that  $\widetilde{H}(u) \equiv 0$ on $\partial \Omega$, while the final assumption ``$0< \kappa <1$'' is satisfied thanks to \eqref{eq:KPR_assumption}.
\end{proof}

We will also use a similar inequality with $\nabla^{(2)}u$ in place of the elliptic operator $Pu$, that follows from \cite[Theorem~3.2 and Remark~3.1]{MR3867979} where we use the nonlinearity $h(t) = t^{\gamma-1}$ and $p=2$. Let us note that we cannot deduce it directly from Lemma~\ref{lemm:KPR_estimate} by using the inequality $|Pu|\le\|\nabla^{(2)}u\|$, because the assumptions on $u$ are different. In particular, the following lemma does not require that $u^{\gamma+1} \in W^{2,1}(\Omega)$.

\begin{lemma}[nonlinear inequality for $\nabla^{(2)}u$, \cite{MR3867979}]\label{withTomasz}
Assume that $\Omega\in C^{0,1}$. If $u\in C(\overline{\Omega})\cap W^{2,1}(\Omega)$, $u\equiv 0$ on $\partial\Omega$, and $u>0$ in $\Omega$, then
  \begin{equation*}
        \int_\Omega \|\nabla u\|^2 u^{\gamma -1} \le C  \int_\Omega \| \nabla^{(2)}u \| u^\gamma ,
    \end{equation*}
    with the constant $C$ not depending on $u$. 
\end{lemma}

Our next lemma applies to the regularity of positive strong solutions to the linear equation $-Pu=F$ where $F$ is bounded and provides a framework for applying Lemma~\ref{lemm:KPR_estimate}. Our particular interest is in the regularity of compositions such as $u^\tau$, where $\tau\ge 1$.

\begin{lemma}[existence and regularity of strong solutions to $-Pu=F$] \label{lemm:KPR_assumption_verification}
    Assume that $\Omega \in C^{1,1}$ and that $\A\in C^{1}(\overline{\Omega})$ satisfies \eqref{ellipticity}. Let $P$ be given by \eqref{eq:P} and let $F\in L^\infty(\Omega)$ be nonnegative almost everywhere,  $F\not\equiv 0$. Then, the problem 
    \begin{align*}
        \begin{cases} -Pu = F\quad &\text{in}\ \Omega,\\
            u \equiv 0\quad &\text{on}\ \partial \Omega.
        \end{cases}
    \end{align*}
    has a unique strong solution $u$ (see Definition~\ref{defStrongSol}), which belongs to $W^{2,p}(\Omega)\cap W^{1,p}_0(\Omega)$ for all $1<p<\infty$. 
    Furthermore, the following statements hold. 
\begin{enumerate}[label={\bf \textup{(\alph*)}}]\label{el-reg} 
        \item \label{lemm:KPR_assumption_verification_i} For any given $p\in (1,\infty)$ we have
        \begin{equation}\label{cald-zyg}
        \| u\|_{W^{2,p}(\Omega)}\le c_p \| F\|_{L^p(\Omega)},    
        \end{equation}
        where the constant $c_p$ does not depend on $u$;
        \item \label{lemm:KPR_assumption_verification_ii} $u\in C^1 (\overline{\Omega})$ and $u\equiv 0$ on $\partial\Omega$ and $u>0$ inside $\Omega$;
        \item \label{lemm:KPR_assumption_verification_iv} We have $u(x) \approx \delta_\Omega(x)$ on $\Omega$, with $d_\Omega (\cdot)$ is as in \eqref{d-omega}.
        \item \label{lemm:KPR_assumption_verification_iii}  For any $\gamma>0$ we have $u^{\gamma+1}\in W^{2,1}(\Omega)$.
        \end{enumerate}
\end{lemma}

\begin{proof}
Since $F\in L^\infty(\Omega)\subseteq L^p(\Omega)$, where $1\le p<\infty$ is arbitrary, the existence of $u$ and the statement \ref{lemm:KPR_assumption_verification_i} follow from the elliptic regularity theory \cite[Theorem~9.15]{MR1814364}.
The Sobolev embeddings with $p$ larger than the dimension $d$, guarantee that $\nabla u$ is continuous up to the boundary (see e.g.~\cite[Exercise 12.59]{Leoni17}), so it is also bounded.
Moreover, $u\equiv 0$ on $\partial\Omega$, because $u\in W^{1,p}_0(\Omega)$.  
Let us prove the strict positivity of $u$.

It is known, that the solution is represented by Green formula, which holds for almost every $x\in\Omega$:
\begin{equation}\label{green0}
 u(x) = \int_\Omega G_\Omega(x,y) F(y)\, dy,
 \end{equation}
where $G_\Omega(\cdot,\cdot)$ is the Green function of the operator $P$ in $\Omega$, see \cite[Theorem~2.3]{MR4554661}.
 The strict positivity of $u$ inside $\Omega$ follows from the nontriviality of $F$ and
  sharp estimates, in particular strict positivity of the Green function:
\begin{align}\label{eq:Greenest2}
&\textnormal{for}\ d=2:\qquad G_\Omega(x,y)\approx \ln \bigg(1 + \frac{\delta_\Omega(x)\delta_\Omega(y)}{\lVert x-y \rVert^2}\bigg),\quad x,y\in \Omega,\\
&\textnormal{for}\ d\geq 3:\qquad 
G_\Omega(x,y)\approx \lVert x-y \rVert^{2-d}\bigg(1\wedge \frac{\delta_\Omega(x)}{\lVert x-y \rVert}\bigg)\bigg(1\wedge \frac{\delta_\Omega(y)}{\lVert x-y \rVert}\bigg),\quad x,y\in \Omega,\label{eq:Greenest}
\end{align}
 see \cite{MR658944}, \cite{MR842803}, \cite{MR1046423}, and \cite[Theorem~1.1]{MR4554661} for the explicit formulation of \eqref{eq:Greenest}.
This gives  the statement \ref{lemm:KPR_assumption_verification_ii}.

Let us prove the statement \ref{lemm:KPR_assumption_verification_iv}.
Since $F$ is nontrivial, there exists an open set $U$, relatively compact in $\Omega$ such that $F$ is also nontrivial on $U$. Therefore, for any open $V\subset\subset \Omega$ such that $U\subset \subset V$,
we have
\begin{align*}
    \bigg(1\wedge \frac{\delta_\Omega(x)}{\lVert x-y \rVert}\bigg)\gtrsim\delta_\Omega(x)\quad \textnormal{and}\quad \frac{\delta_\Omega(x)\delta_\Omega(y)}{\lVert x-y \rVert^2}\approx \delta_\Omega(x),\quad x\in \Omega\setminus V,\ y\in U.
\end{align*} 
Consequently, for $d\geq 3$, by \eqref{eq:Greenest},
\begin{align*}
    u(x) \gtrsim \delta_\Omega(x) \int_U F(y) \lVert x-y \rVert^{2-d}\bigg(1\wedge \frac{\delta_\Omega(y)}{\lVert x-y \rVert}\bigg)\, dy \gtrsim \delta_\Omega(x)\int_U F(y)\, dy \gtrsim \delta_\Omega(x),\quad   x\in \Omega\setminus V, 
\end{align*}
and for $d=2$, since $\ln(1+r) \approx r$ for $r$ close to $0$, by \eqref{eq:Greenest2},
\begin{align*}
    u(x) \gtrsim \delta_\Omega(x)\int_U F(y)\, dy \gtrsim \delta_\Omega(x),\quad x\in \Omega\setminus V.
\end{align*}
Since we also have $\delta_\Omega(x) \approx 1 \approx u (x)$ in $V$, we find that $u(x)\gtrsim \delta_\Omega(x)$ for all $x\in \Omega$. 

 As for the other estimate, $\partial \Omega$  is $C^{1,1}$, $u\in C^1(\overline{\Omega})$, and $u=0$ on $\partial \Omega$, so for $x$ close to the boundary there is a unique $x_0\in \partial \Omega$ such that $\delta_\Omega(x) = |x-x_0|$ and we have $$u(x) = u(x) - u(x_0) \leq \int_0^1 |\nabla u(x_0 + \tau(x-x_0))||x-x_0|\, d\tau  \lesssim \|\nabla u\|_\infty\delta_\Omega(x).$$ Since $u$ is also bounded in  $\Omega$, this proves \ref{lemm:KPR_assumption_verification_iv}.

Let us prove the statement \ref{lemm:KPR_assumption_verification_iii}.
Since $u(x)>0$ for all $x\in \Omega$, we have for all $i,j\in \{ 1,\dots, d\}$
\begin{align}\label{oszac-pierwsza-poch}\partial_{i}(u(x)^{\gamma+1})&= (\gamma+1)u(x)^\gamma \partial_i u(x),\\
\label{oszac-druga-poch}
    \partial_{ij} (u(x)^{\gamma+1}) &= (\gamma+1)\gamma u(x)^{\gamma-1} \partial_i u(x) \partial_j u(x) + (\gamma+1)u(x)^\gamma \partial_{ij}u(x),\\
\label{oszac-norm-druga-poch}
 \|\nabla^{(2)} (u(x)^{\gamma+1})\| &\le  (\gamma+1)\gamma u(x)^{\gamma-1} \|\nabla u(x)\|^2 + (\gamma+1)u(x)^\gamma  \|\nabla^{(2)} u(x)\|,
\end{align}
almost everywhere in $\Omega$. Then, since $u$ and $\nabla u$ are bounded, we have $\|\nabla(u^{\gamma+1})\|\in L^1(\Omega)$. Furthermore, since $u\in W^{2,p}(\Omega)$ for all $p\in [1,\infty)$ and $u\equiv 0$ on $\partial \Omega$, we are in a position to use Lemma~\ref{withTomasz}, which together with the fact that $u$ is bounded implies that 
\begin{align*}
   \int_\Omega  \|\nabla^{(2)} (u(x)^{\gamma+1})\| \lesssim \int_\Omega u^\gamma  \|\nabla^{(2)} u\| \lesssim \|u\|_{W^{2,1}(\Omega)}<\infty.
\end{align*}
This proves \ref{lemm:KPR_assumption_verification_iii}, which ends the proof of the lemma.
\end{proof}
\begin{remark}
\label{hopf-lem}
The usual formulation of the Hopf lemma asserts the positivity of the inward normal derivatives for positive functions $u$ satisfying $-Pu\geq 0$. The estimate $u\gtrsim \delta_\Omega$ following from \ref{lemm:KPR_assumption_verification_iv} easily implies this property; it is also equivalent if $u\in C^1(\overline{\Omega})$, because for such functions the Hopf lemma implies that the inward normal derivative must be uniformly positive by the interior ball condition for $C^{1,1}$ domains.
\end{remark}

We will prove a Kato-type inequality for strong solutions of $-Pu = F$. A different version will be given in Lemma~\ref{lem:Kato} for very weak solutions. In the present statement we assume more regularity on the solution and the right-hand side, but due to that we are able to work with $\A$ which is only once differentiable. Results of this type are known in various settings, see e.g.~\cite{MR4651280}, but we could not find a reference covering our case, so we give a proof. 

\begin{lemma}[a Kato-type inequality for strong solutions]\label{lem:Katoidentity}
    Assume that $\Omega \in C^{1,1}$ and that $\A\in C^{1}(\overline{\Omega})$ satisfies \eqref{ellipticity}. Suppose that $F\in L^\infty(\Omega)$ and that $u\in W^{2,p}(\Omega)\cap W^{1,p}_0(\Omega)$ for all $p\in(1,\infty)$ is a strong solution to \eqref{eq:linearstrong}. Then,    \begin{align}\label{eq:Kato1}
    \int_{\Omega} \A \nabla (u_+) \cdot \nabla \phi + (\Div \A \cdot \nabla (u_+))\phi \leq \int_{\{u>0\}} F\phi,\quad \phi\in H^1_0(\Omega),\ \phi \geq 0.
\end{align}
\end{lemma}
\begin{proof}
     First note that by Lemma~\ref{lemm:KPR_assumption_verification} the strong solution $u$  to \eqref{eq:linearstrong} indeed exists and $u\in W^{2,p}(\Omega)\cap W^{1,p}_0(\Omega)$ for all $p\in(1,\infty)$. 
     For $\epsilon >0$ consider convex $\Phi_\epsilon \in C^2(\R)$ given by
     \begin{align*}
         \Phi_\epsilon(x) = ((x^3 + \epsilon^3)^{1/3} - \epsilon)\textbf{1}_{[0,\infty)}(x).
     \end{align*}
     We have 
     \begin{align*}
         \Phi'_\epsilon(x) &= \frac{x^2}{(x^3 + \epsilon^3)^{\frac 23}}\textbf{1}_{[0,\infty)}(x),\qquad 
         \Phi''_\epsilon(x) = \frac{2x\epsilon^3}{(x^3 + \epsilon^3)^{\frac 53}}\textbf{1}_{[0,\infty)}(x),
     \end{align*}
     therefore $\Phi_\epsilon(0) = 0$, $0\leq \Phi_\epsilon'\leq 1$ and $\Phi'_\epsilon(0) = 0$. Furthermore, \begin{align*}
         \Phi_\epsilon(x) \to x_+,\quad \Phi'_\epsilon(x) \to \textbf{1}_{(0,\infty)}(x),\quad \Phi''_\epsilon(x) \to 0
     \end{align*} pointwise as $\epsilon \to 0$. By the chain rule, \eqref{ellipticity}, and the convexity of $\Phi$ we have
    \begin{align*}
        -P(\Phi_\epsilon(u)) = -\Phi_\epsilon'(u) Pu - \Phi_\epsilon''(u) \A \cdot \nabla u\otimes \nabla u \leq \Phi'_\epsilon(u) F.
    \end{align*}
    Integrating against nonnegative $\phi\in C_c^\infty(\Omega)$ we get
    \begin{align*}
        \int_{\Omega} -P(\Phi_\epsilon(u)) \phi \leq \int_{\Omega} \Phi'_\epsilon(u) F \phi.
    \end{align*}
    The right-hand side converges to $\int_\Omega F\textbf{1}_{\{u>0\}} \phi$, as $\epsilon\to 0$, by the Dominated Convergence Theorem. For the left-hand side we use integration by parts and the Dominated Convergence Theorem:
    \begin{align*}
        \int_{\Omega} -P\Phi_\epsilon(u)\phi &= \int_\Omega \A\nabla \Phi_\epsilon(u)\cdot  \nabla \phi + ((\Div \A) \cdot \nabla \Phi_\epsilon(u))\phi\\
        &= \int_\Omega \Phi_\epsilon'(u)\A\nabla u\cdot  \nabla \phi + ((\Div \A) \cdot \nabla u)\Phi_\epsilon'(u)\phi\\
        &\mathop{\longrightarrow}\limits_{\epsilon\to 0} \int_\Omega \A(\textbf{1}_{\{u>0\}}\nabla u)\cdot  \nabla \phi + ((\Div \A) \cdot (\textbf{1}_{\{u>0\}}\nabla u))\phi.
    \end{align*}
     It is known that $\textbf{1}_{\{u>0\}} \nabla u = \nabla (u_+)$ a.e., see e.g.~\cite[Lemma~7.6]{MR1814364}. Therefore, we get the estimate in \eqref{eq:Kato1} for nonnegative $\phi\in C_c^\infty(\Omega)$. By density, since $u_+ \in H^1_0(\Omega)$, the formula also holds for  $\phi\in H^1_0(\Omega)$ which are nonnegative almost everywhere.
\end{proof}

\bigskip
\subsection{Proof  of Theorem~\ref{maintheorem}}~\\

\subsubsection*{Explanation of the strategy}~\\
The proof consists of first solving the problem with regularized data: the function $f$ and the  nonlinearity $\tau(s) = s^{-\gamma}$. Then we obtain the solution to the original problem as the limit of solutions to regularized problems. We divide the proof into steps.

\bigskip
\subsubsection{Construction of the approximating sequence of the regularized solutions}~\\

We first define the truncation operator for functions $h\colon \Omega\to \R$:
$$
h^{{\langle n \rangle}}(x):= {\rm min}\{ h(x),n\} .
$$ We consider
 the mapping $$S_n\colon H^1_0(\Omega)\to H^1_0(\Omega), \text { where } \psi:=S_nv$$
 is the unique strong solution to the problem (see Lemma~\ref{lemm:KPR_assumption_verification} and \eqref{cald-zyg})
\begin{equation}\label{rownanie}
-P\psi= \frac{f^{{\langle n \rangle}}(x) }{(v^2+\frac{1}{n})^{\frac{\gamma}{2}}} = f^{{\langle n \rangle}}(x)\tau_n (v)=: G_{n}(v), \text{ where } \tau_n (s):= \frac{1}{(s^2+\frac{1}{n})^{\frac{\gamma }{2}}}.
\end{equation}
We will show that
\begin{enumerate}[label={\textup{(\alph*)}}]
    \item \label{main_proof-a} for every fixed $n$, $S_{n}$ is well-defined, continuous on $H^1_0(\Omega)$ and $S_n v$ is positive in $\Omega$;
    \item \label{main_proof-b} $S_n$ admits a fixed point, which is positive in $\Omega$. 
\end{enumerate}

This fixed point will be referred to as $u_n$.

We start with the assertion \ref{main_proof-a}.
The positivity of $S_nv$ and the fact that $S_n$ is well defined as the map into $H^1_0(\Omega)$
follows from Lemma \ref{lemm:KPR_assumption_verification}, because $G_n(v)\geq 0$ is nontrivial and belongs to $L^\infty(\Omega)$, when $v\in H^1_0(\Omega)$.

In order to show that $S_n$ is continuous, we estimate as follows: for $v_1,v_2\in H^1_0(\Omega)$ we have
\begin{align*}
    |G_n(v_1)(x) - G_n(v_2)(x)|^2 &\leq (f^{{\langle n \rangle}}(x))^2\bigg|\frac{1}{(v_1(x)^2 + \tfrac 1n)^{\frac \gamma 2}} - \frac{1}{(v_2(x)^2 + \tfrac 1n)^{\frac \gamma 2}}\bigg|^2\\
    &\leq C_n^2(f^{{\langle n \rangle}}(x))^2 |v_1(x) - v_2(x)|^2,  
\end{align*}
where $C_n$ is the Lipschitz constant of the map $\R \ni t\mapsto (t^2 + 1/n)^{-\gamma/2}$ that depends only on $n$ and $\gamma$. Now, let $\psi_1 = S_n v_1$ and $\psi_2 = S_n v_2$. Then, we have $-P(\psi_2-\psi_1)= G_n(v_2)- G_n(v_1)$ and 
$\|G_n(v_1) -G_n(v_2)\|_{L^2(\Omega)} \leq C \|v_1 - v_2\|_{L^2(\Omega)}$, which together with \eqref{cald-zyg} in
Lemma \ref{lemm:KPR_assumption_verification} implies that $S_n$ is continuous:
\begin{equation*}
    \|\psi_1 - \psi_2 \|_{H^1(\Omega)}\leq\|\psi_1 - \psi_2\|_{H^2(\Omega)} \lesssim \|G_n(v_1) -G_n(v_2)\|_{L^2(\Omega)} \lesssim \|v_1 - v_2\|_{L^2(\Omega)}.
\end{equation*}

Let us now show \ref{main_proof-b}.
By inequality \eqref{cald-zyg} in Lemma \ref{lemm:KPR_assumption_verification}, we deduce that for every $v\in H_0^1(\Omega)$, $\psi = S_n v$ belongs to $H^2(\Omega)$ and satisfies the Dirichlet condition.
Since $\tau_n(v)\le n^{\gamma/2}$, $f^{{\langle n \rangle}}\le n$, by \eqref{cald-zyg}
\begin{equation}\label{to-whole-space}
\| \psi\|_{H^2(\Omega)} \lesssim n^{1+\frac \gamma 2},   
\end{equation}
where the constant depends on $P$, $\Omega$, and $c_2$ from \eqref{cald-zyg}, but not on $n$ and $v$. This, and the compactness of the embedding $H^2(\Omega) \subseteq H^1(\Omega)$, see \cite[Theorem~5.7.1]{MR1625845}, implies that $S_n$ is a compact map on $H^1_0(\Omega)$. Since $S_n$ is continuous on $H^1_0(\Omega)$ and 
maps the space $H^1_0(\Omega)$ into its relatively compact subset, by the Schauder Fixed Point Theorem \cite[Corollary~11.2]{MR1814364} we find that it has a fixed point $u_n\in H^1_0(\Omega)$.

\bigskip
\noindent

\subsubsection{Continuity of the approximate solutions up to the boundary and monotonicity of the sequence at fixed $x$}~\\

By Lemma~\ref{lemm:KPR_assumption_verification} we find that $u_n\in C_0(\Omega)$ for every $n$.
We will now verify that for the functions $u_n$ obtained above the sequence
$(u_n(x))$ is nondecreasing for every $x\in \Omega$.

Let $n>m$ and note that on the set $\{x: u_m(x) > u_n(x)\}$ we have
\begin{align*}
    \frac{f^{{\langle m \rangle}}(x) }{\left(u_m^2(x)+\frac{1}{m}\right)^{\frac{\gamma}{2}}} - \frac{f^{{\langle n \rangle}}(x) }{\left(u_n^2(x)+\frac{1}{n}\right)^{\frac{\gamma}{2}}} \leq 0.
\end{align*}
Then, by Lemma~\ref{lem:Katoidentity}
\begin{equation}\label{eq:Puplus}
\begin{split}
    &\int_\Omega \A\nabla (u_m - u_n)_+ \cdot \nabla \phi + (\Div \A \cdot \nabla (u_m - u_n)_+)\phi\\
    &\leq \int_{\{u_m>u_n\}}\left( \frac{f^{{\langle m \rangle}}(x) }{\left(u_m^2(x)+\frac{1}{m}\right)^{\frac{\gamma}{2}}} - \frac{f^{{\langle n \rangle}}(x) }{\left(u_n^2(x)+\frac{1}{n}\right)^{\frac{\gamma}{2}}} \right) \phi \leq 0
    \end{split}
\end{equation}
holds for every $\phi\in H_0^1(\Omega)$ such that $\phi\geq 0$. 
  The inequality \eqref{eq:Puplus} states that $-P(u_m-u_n)_+ \leq 0$ in a weak sense. We are in a position to use the maximum principle for weak (super-)solutions to divergence form equations, as $P$ is of the form
\begin{align*}
    Pu = \Div (\A \nabla u) - \Div \A\cdot  \nabla u,
\end{align*}
so the assumptions of \cite[Theorem~8.1]{MR1814364} hold. Thus, 
$$\sup\limits_\Omega(u_m-u_n)_+ \leq \sup\limits_{\partial \Omega}(u_m-u_n)_+ = 0,$$ which yields $u_m\leq u_n$.

\smallskip
\noindent
\subsubsection{Local boundedness of the approximate solutions from below}

We now check that the sequence $(u_n)$ is locally bounded from below i.e.~that for every $\omega\subset\subset \Omega$ there exists a constant $c_{\omega}>0$ such that 
$$\inf\{ u_n(x): x\in \omega,\ n\in \mathbb{N}\} \ge c_\omega.$$

Indeed, since $(u_n(x))$ is nondecreasing, for every $\omega\subset\subset \Omega$ and $x\in \omega$ we have
\begin{equation}\label{porownanie}
 u_n(x)\ge u_1(x)  \ge {\rm min}_\omega u_1=: c_\omega >0.
\end{equation}
The strict positivity of $c_\omega$ follows from part \ref{lemm:KPR_assumption_verification_iii} of Lemma~\ref{lemm:KPR_assumption_verification}.

\smallskip
\noindent

\subsubsection{Boundedness of the compositions of the approximate solutions in $H^1_0(\Omega)$}~\\ \label{sec:boundednesscompositions}
We now prove that the functions $w_n:= u_{n}^{\frac{\gamma +1}{2}}$ are bounded in $H^1(\Omega)$. 

Indeed, by Lemma~\ref{lemm:KPR_assumption_verification} we have $u_{n}^{\gamma +1}\in W^{2,1}(\Omega)$, so we are in a position to use Lemma~\ref{lemm:KPR_estimate} to get
\begin{equation}\label{tomasdalimil}
\int_\Omega \|\nabla u_{n}\|_\A^2 u_n^{\gamma -1} \le C
\int_\Omega -Pu_{n} u_{n}^\gamma =C\int_\Omega  f^{{\langle n \rangle}} \frac{u_{n}^\gamma }{\left(u_{n}^2 +\frac{1}{n}\right)^{\frac{\gamma }{2}}}\le C \int_\Omega  f  < \infty,
\end{equation}
where the second step of the calculation holds because $u_n$ is a strong solution, i.e.~we have equality almost everywhere. Furthermore, by Lemma~\ref{lemm:KPR_assumption_verification} we have $\lVert \nabla w_n \rVert^2 = (\frac{\gamma+1}{2})^2u_n^{\gamma-1} \lVert \nabla u_n \rVert^2$ pointwise because $u_n\in C^1 (\overline{\Omega})$. By this and the Poincar\'e inequality, $w_n$ are bounded in $H^1_0(\Omega)$.

\smallskip
\noindent
\subsubsection{Convergence of the compositions}~\\ \label{sec:convergencecompositions} 
We deduce that there exists $w\in H^1_0(\Omega)$ such that
for any $q\in \left[1,\frac{2d}{d-2}\right)$ (for $d=2$ we put $\frac{2d}{2-2} :=\infty$) we have
\begin{equation}\label{slabezbieznosci}
w_{n}\rightarrow w \ {\rm in}\ L^{q}(\Omega)\ {\rm and}\
\nabla w_{n}\rightharpoonup \nabla w \ \hbox{\rm weakly in}\ L^2(\Omega), \ \ {\rm as} \ n\to\infty .
\end{equation}
For this, we first note that there is a measurable almost everywhere finite function $w$ such that we have $w_n \nearrow w$ everywhere in $\Omega$, since the sequence $(w_n)$ is pointwise increasing and bounded in $L^2(\Omega)$. 
Furthermore, by Section~\ref{sec:boundednesscompositions} and the compact Sobolev embeddings \cite[Section~5.7]{MR1625845},  
any subsequence of $(w_{n})$ has a subsequence converging in $L^q(\Omega)$, whenever $q\in [1,\frac{2d}{d-2})$. Since $(w_n)$ converges to $w$ almost everywhere, all these converging subsequences must converge to the same limit $w$ in $L^q(\Omega)$. This implies that the whole sequence converges to $w$ in $L^q(\Omega)$. Indeed, if it was not true, then we would find a subsequence converging in $L^q(\Omega)$ to some $\widetilde{w}\in L^q(\Omega)$ not a.e. equal to $w$, which contradicts the almost everywhere convergence.

A similar argument via contradiction shows that $(w_n)$ must also converge weakly to $w$ in $H^1(\Omega)$.

\smallskip
\noindent

\subsubsection{Convergence of the approximating sequence and regularity}~\\
\label{sec:step6} We will show that for $u:=w^{\frac{2}{\gamma+1}}$,
\begin{align}\label{slabezbieznosci11}
    \textnormal{if}\ \gamma\in [1,\infty):&\quad u_{n}\rightarrow u  \text{ in } L^2(\Omega) \text{ and } \nabla u_{n}\rightharpoonup \nabla u \ \hbox{\rm weakly in}\ L^2_{loc}(\Omega),\\ \label{weak2glob}
    \text{if } \gamma = 1:&\quad \nabla u_{n}\rightharpoonup \nabla u \ \hbox{\rm weakly in}\ L^2(\Omega), \\ \label{weak12}
   \text{if } \gamma \in (1, 2):&\quad \nabla u_{n}\rightharpoonup \nabla u \ \hbox{\rm weakly in}\ L^q(\Omega) \text{ for } q \in \left ( 1, \tfrac{2}{\gamma} \right ), \\ 
    \label{eq:compact01}
    \textnormal{if}\ \gamma \in (0,1):&\quad u_n \to u\ \textnormal{in}\ L^{q_1}(\Omega)\ \textnormal{and}\ \nabla u_n \rightharpoonup \nabla u\ \textnormal{weakly in}\ L^{q_2}(\Omega),
\end{align}
as $n\to \infty$, for all $q_1\in [1,d(\gamma+1)/(d-2))$ and $q_2\in (1,d(\gamma+1)/(d-(1-\gamma))]$.

\textbf{\eqref{slabezbieznosci11} and \eqref{weak2glob}:}
We first show that $u\in L^2(\Omega)$ and $u_n\to u$ in $L^2(\Omega)$. For this, we first note that $u_n\nearrow u$ almost everywhere. The Monotone Convergence Theorem gives 
 $\int_\Omega u_n^2 \to \int_\Omega u^2$, but perhaps that limit is infinite. To explain that the limit is finite, we observe that
\begin{align*}
\int_\Omega u^2 = \lim\limits_{n\to\infty} \int_\Omega u_n^2  = \lim\limits_{n\to\infty} \int_\Omega w_n^{\frac{4}{\gamma+1}}. 
\end{align*}
When $\gamma =1$ we have $u_n=w_n$, so the finiteness of the limit follows from Section~\ref{sec:convergencecompositions} directly. When $\gamma >1$, the H\"older inequality applied with $s=\frac{\gamma +1}{2}\in (1,\infty)$ gives
\begin{align*}
\int_\Omega w_n^{\frac{4}{\gamma +1}} \le \left( \int_\Omega w_n^{\frac{4}{\gamma +1}\cdot s}\right)^{\frac{1}{s}}|\Omega|^{1-\frac{1}{s}} = \left( \int_\Omega w_n^2\right)^{\frac{2}{\gamma +1}}|\Omega|^{\frac{\gamma -1}{\gamma +1}}\stackrel{n\to\infty}{\rightarrow}  \left( \int_\Omega w^2\right)^{\frac{2}{\gamma +1}}|\Omega|^{\frac{\gamma -1}{\gamma +1}},
\end{align*}
where the last convergence also follows from Section~\ref{sec:convergencecompositions}. We have shown that $\int_\Omega u_n^2 \nearrow \int_\Omega u^2<\infty$ and $u_n \nearrow u$ almost everywhere, which implies that $u_n\to u$ in $L^2(\Omega)$ by the Dominated Convergence Theorem.

We will show that the sequence $(u_n)$ is bounded in $H^1_{loc}(\Omega)$. Indeed, as $u_n$ are strictly positive in $\Omega$, we have
\begin{align}\label{eq:gradientu}
\nabla u_n =\nabla (w_n^{\frac{2}{\gamma +1}}) = \frac{2}{\gamma +1} w_n^{\frac{1-\gamma}{\gamma +1}}\nabla w_n.
\end{align}
Since the sequence $(w_n^{\frac{1-\gamma}{\gamma +1}})$ is locally uniformly bounded in $\Omega$, we deduce that
the sequence $(\nabla u_n)$ is bounded in $L^2(\omega)$ for all $\omega \subset\subset \Omega$. Putting this together with the boundedness of $(u_n)$ in $L^2(\Omega)$ established above, we obtain that $(u_n)$ is bounded in $H^1(\omega)$. As before, by using the compactness argument
we get that every subsequence has a weakly convergent subsequence in $H^1(\omega)$. At the same time, the same argument as in \eqref{eq:gradientu} shows that $u \in H^1(\omega)$ and we also have $u_n\rightarrow u$ in $L^2(\omega)$, whence it follows that 
\begin{equation}\label{slabazbloc}
\nabla u_n\rightharpoonup \nabla u\ \ \hbox{weakly in}\ \  L^2(\omega),
\end{equation}
which implies \eqref{slabezbieznosci}. 

Note that for $\gamma=1$ the convergence is global in $\Omega$, i.e.~$u \in H^1(\Omega)$ and 
\begin{equation*}
\nabla u_n\rightharpoonup \nabla u\ \ \hbox{weakly in}\ \  L^2(\Omega).
\end{equation*}

\textbf{\eqref{weak12}:} For $\gamma > 1$, we note that Lemma~\ref{lemm:KPR_assumption_verification} shows that $u_n \gtrsim \delta_{\Omega}$, and since the sequence $u_n$ is pointwise increasing, the comparability constant does not depend on $n$. Hence, $w_n^{\frac{1-\gamma}{\gamma +1}} = u_n^{\frac{1-\gamma}{2}} \lesssim \delta_{\Omega}(x)^{\frac{1-\gamma}{2}}$. Since $\Omega$ is Lipschitz, $\delta_{\Omega}(x)^{\alpha} \in L^1(\Omega)$ for $\alpha > -1$. Thus, applying the H\"older inequality in \eqref{eq:gradientu}, we get that for $q < \frac{2}{\gamma} < 2$
\begin{equation*}
\begin{split}
    \int_{\Omega} \lVert \nabla u_n \rVert^q &= \left ( \frac{2}{\gamma +1} \right )^q \int_{\Omega} \lVert w_n^{\frac{1-\gamma}{\gamma +1}}\nabla w_n \rVert^q \\
    &\leq \left ( \frac{2}{\gamma +1} \right )^q \left( \int_{\Omega} \lVert \nabla w_n \rVert^2  \right)^{\frac{q}{2}} \left( \int_{\Omega} w_n^{\frac{1-\gamma}{\gamma +1}\frac{2q}{2-q}}  \right)^{\frac{2-q}{2}} \\
    &\lesssim \left ( \frac{2}{\gamma +1} \right )^q \left( \int_{\Omega} \lVert \nabla w_n \rVert^2  \right)^{\frac{q}{2}} \left( \int_{\Omega} \delta_{\Omega}(x)^{\frac{1-\gamma}{2} \frac{2q}{2-q}}  \right)^{\frac{2-q}{2}} 
\end{split}
\end{equation*}
An easy calculation then shows that $\frac{1-\gamma}{2} \frac{2q}{2-q} > -1$ if and only if $q < \frac{2}{\gamma}$. Consequently, $(\nabla u_n)$ is a bounded sequence in $L^q(\Omega)$ for all $q < \frac{2}{\gamma}$. Finally, when $\gamma < 2$ and $q \in (1, \frac{2}{\gamma})$, the reflexivity of $W^{1,q}(\Omega)$ implies that every subsequence of $(\nabla u_n)$ has a weakly convergent subsequence. However, we already know that $u_n \to u$ in $L^{2}(\Omega)$, whence
\begin{equation}\label{slabaqglob}
\nabla u_n\rightharpoonup \nabla u\ \ \hbox{weakly in}\ \  L^q(\Omega).
\end{equation}

\textbf{\eqref{eq:compact01}:} First, note that \eqref{slabezbieznosci} implies that $u_n\to u$ in $L^{q_1}(\Omega)$ for all $q_1\in [1,d(\gamma+1)/(d-2))$; indeed, this follows from the Dominated Convergence Theorem as $u_n \nearrow u \in L^{q_1}(\Omega)$ for $q_1 < d(\gamma+1)/(d-2))$. We will now estimate $\lVert \nabla u_n \rVert^{q_2}$, where for now $q_2\in(0,2)$. By \eqref{eq:gradientu} and the Young inequality:\begin{align*}
    \lVert \nabla u_n \rVert^{q_2} \lesssim w_n^{\frac{2q_2}{2-q_2}\frac{1-\gamma}{\gamma+1}} + \lVert \nabla w_n \rVert^2. 
\end{align*}
By Section~\ref{sec:boundednesscompositions} and 
the Sobolev inequality applied to $w_n\in H^{1}(\Omega)$, the right-hand side above has uniformly bounded integrals, provided that
\begin{align*}
    \frac{2q_2}{2-q_2}\frac{1-\gamma}{\gamma+1} \leq \frac {2d}{d-2} \quad &\iff \quad q_2\leq\frac{d(1+\gamma)}{d-1+\gamma},\qquad d>2,\\
    \frac{2q_2}{2-q_2}\frac{1-\gamma}{\gamma+1} <\infty \quad &\iff \quad q_2 <2,\qquad\qquad\ \ \quad d=2.
\end{align*}
Note that $1 < \frac{d(1+\gamma)}{d-1+\gamma} < 2$. Therefore, by the weak compactness of closed balls in $L^{q_2}(\Omega)$, $q_2\in (1,d(\gamma+1)/(d-(1-\gamma))]$, and the arguments similar to those in Section~\ref{sec:convergencecompositions}, we get \eqref{eq:compact01}. Note that the convergence of $u_n$ in the range of $L^{q_1}(\Omega)$ given in \eqref{eq:compact01} also follows from the Rellich--Kondrachov theorem as $(d(\gamma+1)/(d-(1-\gamma)))^* = d(\gamma+1)/(d-2)$, where $q^* = dq/(d-q)$ is the critical Sobolev exponent.

\bigskip
\subsubsection{The limit is a weak solution}~\\
We now check that $u$ is a weak solution to
$$
-Pu = \frac{f}{u^\gamma}.
$$

Since $u_n \in W^{2,p}(\Omega)$ for all $p\in (1,\infty)$, by integrating the strong formulation against a test function $\phi\in C_c^\infty(\Omega)$ and using the integration by parts formula \cite[Theorem~1.5.3.1]{MR775683}, we get
\begin{align*}
     \int_\Omega  \frac{f^{{\langle n \rangle}} }{\left(u_{n}^2 +\frac{1}{n}\right)^{\frac{\gamma }{2}}}\phi = \int_\Omega -Pu_{n} \phi = \int_\Omega  ({\A}\nabla u_n)\cdot  \nabla \phi +\int_\Omega (\Div{\A} \cdot \nabla u_n ) \phi =: \textit{I}_n+\textit{II}_n.
\end{align*}
 Recall that the expression on the right-hand side is $\langle -Pu_n, \phi \rangle$ in the sense of \eqref{dist-special}.  

Let $\omega\subset\subset\Omega$ be an open set containing $\supp \phi$. Since $\A$, $\Div \A$, $\phi$, and $\nabla \phi$ are bounded, using \eqref{slabezbieznosci11} and \eqref{eq:compact01} we find that
\begin{align*}
\textit{I}_n = \int_\omega (\A \nabla u_n)\cdot \nabla \phi &\mathop{\longrightarrow}\limits_{n\to\infty} \int_\omega   (\A \nabla u) \cdot \nabla \phi= \int_\Omega  (\A \nabla u)\cdot \nabla \phi, \\
\textit{II}_n = \int_\omega (\nabla u_n\cdot \Div \A) \phi &\mathop{\longrightarrow}\limits_{n\to\infty} \int_\omega (\nabla u\cdot \Div \A) \phi = \int_\Omega (\nabla u\cdot \Div \A) \phi.
\end{align*}
As a consequence, using the notation of \eqref{dist-special} we get
\begin{align*}
&\langle -Pu_n, \phi \rangle \mathop{\longrightarrow}\limits_{n\to\infty} \langle -Pu, \phi \rangle. 
\end{align*}
On the other hand, since $u_n,u\geq c_\omega>0$ in $\omega$ and $f^{{\langle n \rangle}}\leq f\in L^1(\Omega)$, by the Dominated Convergence Theorem we obtain
$$
\int_\omega  \frac{f^{{\langle n \rangle}}(x)}{\left(u_n^2(x) +\frac{1}{n}\right)^{\frac{\gamma}{2}}} \phi(x)\, dx \mathop{\longrightarrow}\limits_{n\to\infty} \int_\omega  \frac{f(x)}{u^\gamma(x)} \phi(x)\, dx. 
$$
Therefore
$$
\langle -Pu, \phi \rangle = \int_\omega  \frac{f(x)}{u^\gamma(x)} \phi(x)\, dx,$$
which proves that $u$ is a weak solution.
\subsubsection{Trace and regularity of $u$} The Sobolev regularity properties \ref{item:gammageq2}--\ref{item:gammaleq1} follow from the estimates \eqref{slabezbieznosci11}--\eqref{eq:compact01} in Section~\ref{sec:step6}. Thus, to conclude the proof it remains to verify that the trace condition \eqref{eq:Dirichletcond} is satisfied. We first show that it holds for functions in $W^{1,1}_0(\Omega)$. This fact seems to be folklore, but we could not locate an explicit proof, so we give the details for the convenience of the reader. If $v\in W^{1,1}_0(\Omega)$, then by the discussion at the beginning of Section~\ref{sec:lemmas} we may assume that $u=0$ everywhere on $\partial \Omega$. Furthermore, the following is true:
\begin{align}\label{eq:tracechain}
    \lim\limits_{\epsilon\to 0}\frac 1 \epsilon \int_{\Omega_\epsilon} |v| =\lim\limits_{\epsilon\to 0}\frac 1 \epsilon \int_0^\epsilon \int_{\{ x\in \Omega: \delta_\Omega (x) =t \} } |v| \, d\sigma (x)\, dt
\leq  \lim\limits_{\epsilon\to 0}\frac 1 \epsilon \int_0^\epsilon \bigg( \int_{\Omega_t} \|\nabla v\|\, dx  \bigg) dt =0.
\end{align}
The first equality follows from the coarea formula  \cite[Theorem~3.2.12]{MR257325}, as $\delta_\Omega$ is Lipschitz and $\|\nabla \delta_\Omega\|=1$ on $\Omega_\epsilon$ for small $\epsilon$. For the (only) inequality in \eqref{eq:tracechain} we note that since $\partial \Omega$ is $C^{1,1}$, it satisfies the interior ball condition, therefore for $\epsilon$ small enough for every $x\in \Omega_\epsilon$ there is exactly one $\Pi(x)\in \partial \Omega$ such that $\|x-\Pi(x)\| = \delta_\Omega(x)$. Thus, for $\delta_\Omega(x)=t$ small enough, since $v\equiv 0$ on $\partial \Omega$, we have 
\begin{align}\label{eq:vtracebound}
|v(x)| = |v(x) - v(\Pi(x))| \leq \int_0^{t} \|\nabla v (\Pi(x) + \frac{s}{t} (x-\Pi(x))\|\, ds. 
\end{align}
Let $\vec{n}(x)$ be the inner normal to $\partial \Omega$ at $\Pi(x)$ and note that
\begin{align*}
    \Pi(x) + \frac{s}{t} (x-\Pi(x)) = \Pi(x) + s\vec{n}(x).
\end{align*}
Therefore, since $\partial\Omega$ is $C^{1,1}$, the mappings \begin{align*}
    T\colon \{x\in \Omega: \delta_\Omega(x) =t\}\to \{x\in \Omega: \delta_\Omega(x) =s\},\quad T(x)= \Pi(x) + s\vec{n}(x) \end{align*} and their inverses are uniformly Lipschitz for $s\in [0,t]$. Therefore, by integrating \eqref{eq:vtracebound} over $x$, using integration by substitution and the coarea formula,
\begin{align*}
    \int_{\{ x\in \Omega: \delta_\Omega (x) =t \} } |v(x)| \, d\sigma (x) &\leq \int_0^t \int_{\{ x\in \Omega: \delta_\Omega (x) =t \} }\|\nabla v (\Pi(x) + \frac{s}{t}(x-\Pi(x))\|\, d\sigma (x) \, ds\\
    &\approx \int_0^t \int_{\{x\in \Omega: \delta_\Omega(x) = s\}} \|\nabla v(x)\|\, d\sigma(x)\, ds
    = \int_{\Omega_t} \|\nabla v\|\, dx,
\end{align*}
which implies the inequality in \eqref{eq:tracechain}. The last equality in \eqref{eq:tracechain} holds because $\|\nabla u\|\in L^1(\Omega)$, so $\int_{\Omega_\epsilon} \|\nabla v\| \to 0$ as $\epsilon\to 0$. Therefore \eqref{eq:Dirichletcond} holds for functions in $W^{1,1}_0(\Omega)$.

Then, for $\gamma\in (0,1]$ the solution $u$ obtained above satisfies the trace condition, as $u\in W^{1,q}_0(\Omega)$ for some $q>1$. For $\gamma>1$ we have $u^{\frac{\gamma+1}2} \in H^1_0(\Omega)$, so $u^{\frac{\gamma+1}2}$ has trace \eqref{eq:Dirichletcond} equal to 0, but since $\frac{\gamma+1}2>1$, this implies that $u$ also has trace 0 by the Jensen inequality.

Finally, since $u$ was obtained as an increasing sequence of $u_n$ and $u_1\gtrsim \delta_\Omega$, we find that $u\gtrsim \delta_\Omega$.

This ends the proof of Theorem~\ref{maintheorem}. \qed

\section{Proof of Theorem \ref{uniqregu}. The uniqueness result}\label{sec:uniquenessproof}
\subsection{Some facts for linear equations}
We will repeatedly consider solutions to auxiliary linear equations involving $P$ and $P^*$. As $P$ was covered in the previous section, here we give similar results for the equation $-P^* v = F$. Recall that the adjoint operator $P^*$ has the divergence form:
\begin{align*}
    P^*v = \Div(\A \nabla v + \Div (\A) v).
\end{align*}
Assuming $\A\in C^2(\overline{\Omega})$, we may compute a reformulation of $P^*$ in nondivergence form:
\begin{align*}
        P^*v = \Div(\A \nabla v) + \Div \A \cdot \nabla v + (\Div(\Div \A))v,
\end{align*}
 whence  
\begin{align} \label{eq:P*non-div-P}
    P^*v = Pv + 2\Div \A \cdot \nabla v + (\Div(\Div \A))v,
\end{align}
with the equalities understood in the a.e.~sense for $v\in W^{2,1}_{loc}(\Omega)$ and $P$ as in \eqref{eq:P}.
Hence, we have a possibly signed zero-order term, which means that we cannot directly apply the strong solution framework of \cite[Chapter~9]{MR1814364}. Nevertheless, it is possible to obtain similar regularity results by considering the weak solutions and bootstrapping the regularity as we demonstrate below. We say that $v\in H^1_0(\Omega)$ is a weak solution to
\begin{align}\label{eq:Pstarv}
\begin{cases}
    -P^*v = F\quad &{\rm in}\ \Omega,\\
    v\equiv 0\quad &{\rm on}\ \partial\Omega,
\end{cases}
\end{align}
if the following equality holds
\begin{align*}
    \int_\Omega (\A \nabla v \cdot \nabla \phi + v\Div \A  \cdot \nabla \phi) = \int_\Omega F\phi,\quad \phi\in H^1_0(\Omega).
\end{align*}
\begin{lemma}[existence and regularity of solutions to \eqref{eq:Pstarv}]\label{lem:linear}
    Assume that $\Omega \in C^{0,1}$, $\A\in C^1(\overline{\Omega})$ satisfies \eqref{ellipticity}, and $F\in L^2(\Omega)$. Then,
    \begin{enumerate} 
        \item \label{lem:linear_i} There exists exactly one weak solution $v\in H^1_0(\Omega)$ to \eqref{eq:Pstarv}.
         \item  \label{lem:linear_ii} If in addition $\Omega \in C^2$, $\A\in C^2(\overline{\Omega})$, and $F\in L^\infty(\Omega)$, then
        \begin{enumerate}[label={\bf \textup{(\alph*)}}]
            \item\label{lem:linear_iia} $v\in C^1(\overline{\Omega})\cap W^{2,p}(\Omega)$ for every $p\in (1,\infty)$ and \eqref{eq:Pstarv} holds almost everywhere in $\Omega$.
            \item\label{lem:linear_iib} We have $v\lesssim \delta_\Omega$ with the comparability constant depending only on $d, c_\A,C_\A,\Omega$, $\|A\|_{C^1(\overline{\Omega})}$, and $\|F\|_\infty$.
            \item\label{lem:linear_iic} If $F\geq 0$ and $F>0$ on a set of positive Lebesgue measure, then $v\gtrsim \delta_\Omega$.
        \end{enumerate}
    \end{enumerate}
\end{lemma}
\begin{proof}
     \ref{lem:linear_i} For the existence and uniqueness of the weak solution $v\in H^1_0(\Omega)$ we use \cite[Theorem~8.3]{MR1814364} given for operators of the form
     $$Lv = \Div (\A\nabla v + b v) + c\nabla v + dv.$$ 
    Aside from uniform ellipticity of $\A$ and boundedness of $\A$, $b$, $c$ and $d$, which hold in our case as $b=\Div \A$, $c=0$, and $d=0$, the result requires that
     \begin{align*} 
         \int_\Omega(d\phi - b\cdot D\phi) \geq 0,\quad \phi\geq 0,\ \phi\in C^1_c(\Omega).
     \end{align*}
     However, as explained in the comment before Section~8.3 on \cite[p.~183]{MR1814364}, this condition can be replaced by 
     \begin{align*} 
         \int_\Omega (d\phi + c\cdot D\phi) \leq 0,\quad \phi\geq 0,\ \phi\in C^1_c(\Omega),
     \end{align*}
     which holds trivially for $P^*$. Thus we get the existence and uniqueness of a weak solution $v\in H^1_0(\Omega)$. \medskip

     \noindent \ref{lem:linear_ii} \ref{lem:linear_iia}:
    Assuming that $\Omega \in C^2$ and $\A\in C^2(\overline{\Omega})$, we are in a position to apply \cite[Theorems~8.8 and 8.12]{MR1814364} to find that $v\in W^{2,2}(\Omega)$ and that $v$ satisfies $-P^*v = F$ almost everywhere.

    To further improve the regularity, note first that by the Sobolev inequality we have $v \in W^{1,2d/(d-2)}(\Omega)$ (we only consider $d>2$ in this argument, $d=2$ is simpler). We use \eqref{eq:P*non-div-P} to reformulate $-P^*v = F$ as
    \begin{align}\label{eq:pxiequation}
       -Pv = (\Div(\Div \A)) v + 2\Div \A \cdot \nabla v + F.
    \end{align}
    Then, the right-hand side belongs to $L^{2d/(d-2)}(\Omega)$, therefore by \cite[Theorem~9.15]{MR1814364} we find that $v \in W^{2,2d/(d-2)}(\Omega)$. Let us now denote 
    \begin{align*}
        p_0 &:= 2d/(d-2), \qquad
        p_{n+1} := \begin{cases}
            (p_n)^*= \frac{dp_n}{d-p_n} &\text{ if } p_n < d, \\
            \infty &\text{ if } p_n \geq d,
        \end{cases}
    \end{align*}
    and note that $\frac{p_{n+1}}{p_n} \geq \frac{d}{d - p_0} > 1$, whence $p_n$ is unbounded and so eventually infinite. Using the Sobolev inequality again, we find that the right-hand side of \eqref{eq:pxiequation} belongs to $L^{p_1}(\Omega)$, and then \cite[Theorem~9.15]{MR1814364} implies that $v \in W^{2,p_1}(\Omega)$. By repeatedly using the Sobolev inequality and \cite[Theorem~9.15]{MR1814364} we find that the right-hand side of \eqref{eq:pxiequation}
    belongs to $L^\infty$, 
    therefore $v \in W^{2,p}(\Omega)$ for all $p\in (1,\infty)$. The fact that $u\in C^1(\overline{\Omega})$ follows from the Sobolev embeddings. \medskip
    
    \noindent \ref{lem:linear_ii} \ref{lem:linear_iib}: Since $v\in W^{2,p}(\Omega)$ satisfies $-P^*v = F$ almost everywhere, integration by parts and \cite[Lemma~1.7 and Theorem~1.9]{MR4056785} (see also \cite{MR4301407} for $d=2$) imply that it can be represented in terms of the Green function $G^*_\Omega$ of $P^*$. But since $G^*_\Omega(x,y) = G_\Omega(y,x)$, we have 
    \begin{align*}
        v(x) = \int_\Omega G_\Omega(y,x)F(y)\, dy,\quad x\in \Omega.
    \end{align*}
    By the sharp estimates \eqref{eq:Greenest} and \eqref{eq:Greenest2} we have $G_\Omega(x,y) \approx G_\Omega(y,x)$, therefore 
    \begin{align*}
        |v(x)| \lesssim \int_\Omega G_\Omega(x,y)\|F\|_\infty\ dy = \zeta \|F\|_\infty,
    \end{align*}
    where $\zeta$ solves $-P\zeta = 1$. By Lemma~\ref{lemm:KPR_assumption_verification} we have $\zeta \approx \delta_\Omega$, therefore \textbf{(b)} holds.\medskip
    
    \noindent \ref{lem:linear_ii} \ref{lem:linear_iic}: is proved in the same way as the lower bound in Lemma~\ref{lemm:KPR_assumption_verification} \ref{lemm:KPR_assumption_verification_iv}, we skip the details.
\end{proof}

\subsection{First extension of the class of test functions}\label{sec:firstextension}
We will show that the class of the admissible test functions for the distributional solutions of \eqref{pe} can be extended to the set
\begin{align}\label{eq:firstextension}
    \mathcal{A} = \{\phi\in C^1_c(\Omega)\cap W^{2,1}(\Omega): P^*\phi\in L^\infty(\Omega)\}.
\end{align}
The class will be extended even further, but this step is needed to obtain the weighted integrability of $f/u^\gamma$, which will enable us to consider test functions which vanish at $\partial \Omega$, but are not necessarily compactly supported.

In order to make the first extension, we will use a Friedrichs-type lemma. The following statement might be known, but we were unable to locate a satisfactory formulation in the literature, so we give a proof for the sake of completeness. See e.g.~\cite[Corollary~3.3]{MR1732050} for a related result for the first-order derivatives.

Recall that $\varphi_\epsilon\in C_c^\infty(B_\epsilon)$ is the standard  mollifier as in Section \ref{general-notation}.

\begin{lemma}[Friedrichs-type lemma]\label{lem:Friedrichs}
    Let $Lv = \A \cdot \nabla^{(2)}v + b\cdot\nabla v + cv$ with $\A \in C^1(\Omega)$ and $b,c$ bounded and measurable. Suppose that $v\in C_c(\Omega)\cap W^{1,\infty}(\Omega)\cap W^{2,1}(\Omega)$ and $Lv\in L^\infty(\Omega)$. Then, there exist $C>0$ and $\epsilon_0>0$ such that for all $\epsilon\in (0,\epsilon_0)$
    \begin{align*}
        \|L(v\ast \varphi_\epsilon)\|_\infty \leq  C(\|Lv\|_\infty + \|v\|_{W^{1,\infty}(\Omega)})
    \end{align*}
\end{lemma}
\begin{proof}
    Since $b$ and $c$ are bounded, we have 
    \begin{align*}
        |(b\cdot\nabla v)\ast \varphi_\epsilon + (cv)\ast\varphi_\epsilon| \leq (\|b\|_\infty + \|c\|_\infty)\|v\|_{W^{1,\infty}(\Omega)}.
    \end{align*}
    Therefore it remains to estimate the second-order term. We have by the Dominated Convergence Theorem and integration by parts that
    \begin{equation*}
        \A \cdot (\nabla^{(2)}(v\ast \varphi_\epsilon))(x) = \A(x) \cdot \int_{B(x,\epsilon)} \nabla^{(2)}v(y) \varphi_\epsilon(x-y) \, dy,
    \end{equation*}
    whence    
    \begin{align*}
        |\A \cdot (\nabla^{(2)}(v\ast \varphi_\epsilon))(x)| \leq& |\varphi_\epsilon\ast (\A \cdot \nabla^{(2)}v)(x)| \\ 
        &+ |\A \cdot (\nabla^{(2)}(v\ast\varphi_\epsilon))(x) - \varphi_\epsilon\ast (\A \cdot \nabla^{(2)}v)(x)|\\
        =&\|Lv\|_\infty + (\|b\|_\infty + \|c\|_\infty)\|v\|_{W^{1,\infty}(\Omega)} \\
        &+ \bigg|\int_{B(x,\epsilon)} (\A(x) - \A(y)) \cdot \nabla^{(2)}v(y) \varphi_\epsilon(x-y)\, dy\bigg|.
    \end{align*}
      We integrate by parts in the last integral and use the triangle inequality, getting
    \begin{align*}
        &\bigg|\int_{B(x,\epsilon)} (\A(x) - \A(y))\cdot\nabla^{(2)}v(y) \varphi_\epsilon(x-y)\, dy\bigg| \\ &\leq\int_{B(x,\epsilon)} |(\A(x) - \A (y))\nabla v(y) \cdot \nabla \varphi_\epsilon(x-y)|\, dy 
        + \int_{B(x,\epsilon)} \|\A\|_{C^1(\omega)}|\nabla v(y)|\varphi_\epsilon(x-y)\, dy\\
        &\leq \|v\|_{W^{1,\infty}(\Omega)}\|\A\|_{C^1(\omega)}\bigg(\int_{B(0,\epsilon)} \|x-y\|\epsilon^{-d-1} \nabla \varphi_1(y/\epsilon)\, dy + 1\bigg) 
        \leq C\|v\|_{W^{1,\infty}(\Omega)},
    \end{align*}
    where $\omega$ is a relatively compact subdomain of $\Omega$ depending only on $\epsilon_0$ and $\supp v$. This ends the proof.
\end{proof}

We may now move on to the actual extension.

\begin{lemma}[first extension of the class of test functions]\label{lem:firstextension}
    Let $\A\in C^2(\Omega)$ and assume that $u\in L^1(\Omega)$ and $F\in L^1_{loc}(\Omega)$ satisfy $-Pu = F$ in the distributional sense, that is,
    \begin{align*}
        -\int_\Omega uP^* \phi = \int_\Omega F\phi,\quad \phi \in C^\infty_c(\Omega).
    \end{align*}
    Then,
    \begin{align*}
        -\int_\Omega uP^* \phi = \int_\Omega F\phi,\quad \phi\in \mathcal{A},
    \end{align*}
    where $\mathcal{A}$ is as in \eqref{eq:firstextension}. 
\end{lemma}
\begin{proof}
    Take $\phi\in \mathcal{A}$. Then, since $\phi\in C^1_c(\Omega)$, for sufficiently small $\epsilon>0$ we have $\phi\ast \varphi_\epsilon \in C^\infty_c(\Omega)$. By the Dominated Convergence Theorem, $\int F\phi\ast \varphi_\epsilon \to \int F\phi$. Furthermore, as $\phi\in W^{2,1}_{loc}(\Omega)$, we have $P^*(\phi\ast\varphi_\epsilon) \to P^*\phi$ almost everywhere as $\epsilon\to 0$. By the definition of $\mathcal{A}$ we have $P^*\phi\in L^\infty(\Omega)$, therefore the statement of the lemma follows from Lemma~\ref{lem:Friedrichs} and the Dominated Convergence Theorem.
\end{proof}
\subsection{Weighted integrability of $f/u^\gamma$}\label{sec:weightedintegrability}

For the purpose of establishing the weighted integrability of $f/u^\gamma$, we will use the special test function $\xi$, a strong solution to
\begin{align*}
\begin{cases}
    -P^\ast \xi=1\quad &{\rm in}\ \Omega,\\
    \xi = 0\quad &{\rm on}\ \partial \Omega.
    \end{cases}
\end{align*}
The existence and uniqueness of $\xi$ is given by Lemma~\ref{lem:linear}, whence we also have $\xi\in C^1(\overline{\Omega})\cap W^{2,p}(\Omega)\cap W^{1,p}_0(\Omega)$ for all $p\in (1,\infty)$, $\xi>0$ in $\Omega$, and $\xi \approx \delta_{\Omega}$.  

Since $\xi$ cannot be used directly as a test function, we will approximate it by compactly supported functions, as in the following result.

\begin{lemma}[approximation of $\xi$]\label{lem:specialxi}
    Assume that $\Omega \in C^{2}$ and that $\A\in C^{2}(\overline{\Omega})$ satisfies \eqref{ellipticity}. Then there exists an increasing sequence of positive functions $\xi_n\nearrow \xi$ such that $\xi_n \in \mathcal{A}$ (see \eqref{eq:firstextension}) 
    and for almost every $x\in \Omega$
    $$
        -P^*\xi_n (x) \le D_{\A},
    $$
    where  $D_{\A}=  1+ \| (\Div(\Div\A ))_+\|_{\infty}\|\xi\|_\infty$.
\end{lemma}
\begin{proof}
    \smallskip
    \noindent
    Let us start by the construction of $\xi_n$.
    
     Consider a convex function $\Phi\in C^\infty (\mathbb{R})$ such that $\Phi$ is nondecreasing, $\Phi(x)= 0$ for $x<\delta$ for some $\delta>0$, $\Phi$ is Lipschitz with constant $1$ (in particular  $\Phi(t)\le t$ and $\Phi' \leq 1$),  
     and $\frac{\Phi(t)}{t}\nearrow 1$ as $t\to \infty$.

     As an example function $\Phi$ we can consider $\Phi (t):= \varphi_\epsilon * {\rm max}\{ 0, t-1\}$, where $\varphi_\epsilon$ is the standard mollifier.

    Define 
    $$
        \xi_n(x):= \frac{1}{n}\Phi (n\xi (x)) .
    $$
    By the regularity and decay of $\xi$ and the definition of $\Phi$ we have $\xi_n\in C^1_c(\Omega)$. 
    Furthermore, 
    $$
        \xi_n (x) = \frac{\Phi (n\xi (x))}{n\xi (x)}\xi (x) \le \xi(x).
    $$
    and $\Phi(n\xi)/n\xi$ converge to 1 as $n\to \infty$, and increase since $\Phi$ is convex. This implies that 
    $
        \xi_n(x)\nearrow \xi (x)
    $
    for every $x\in \Omega$.

    We will show that
     the following estimate holds almost everywhere:
    \begin{align*}
    - P^*\xi_n &\le \Phi^{'}(n\xi) + \Div(\Div\A)(\Phi^{'}(n\xi)\cdot\xi-\frac{1}{n}\Phi(n\xi)).
    \end{align*}
    
    Indeed, $\xi_n\in W^{2,1}_c(\Omega)$ and
    \begin{align}
        \nabla \xi_n   &=\Phi^{'}(n\xi )\nabla\xi  ,\nonumber\\
        \nabla^{(2)}\xi_n  &= \Phi^{'}(n\xi )\nabla^{(2)}\xi  + n\Phi^{''}(n\xi )\nabla\xi  \otimes \nabla\xi  ,\nonumber\\
        \Div(\A\nabla\xi_n) &= (\Div\A)\cdot \nabla\xi_n + \A\cdot \nabla^{(2)}\xi_n\label{1oszac}\\
        &= \Phi^{'}(n\xi )(\Div\A)\cdot \nabla\xi   + \Phi^{'}(n\xi ) \A\cdot\nabla^{(2)}\xi  \nonumber\\ 
        & \quad + n\Phi^{''}(n\xi ) \A \nabla\xi  \cdot \nabla\xi  \nonumber\\
        &\ge  \Phi^{'}(n\xi )\Div(\A \nabla\xi ) ,\nonumber
    \end{align}
    where we have used the convexity of $\Phi$ and \eqref{ellipticity}, which imply $\Phi^{''}(n\xi )(\A \nabla\xi \cdot\nabla \xi)\ge 0$.
    We also have
    \begin{align}\label{eq:divdivterm}
    \Div(\Div\A\xi_n) &= \Div(\Div\A)\xi_n + (\Div\A)\cdot\nabla \xi_n\\
    &= \Div(\Div\A)\frac{1}{n}\Phi(n\xi ) + 
     \Phi^{'}(n\xi )(\Div\A)\cdot \nabla\xi  \nonumber. 
    \end{align}
    Consequently,
    \begin{align*}
    -P^*\xi_n &= -\Div(\A\nabla\xi_n) - \Div(\Div\A\xi_n)\\
    &\! \!\stackrel{\eqref{1oszac} }{\le}  - \Phi^{'}(n\xi )[\Div(\A\nabla\xi )]  -  \Phi^{'}(n\xi ) \Div\A\cdot \nabla\xi   - \Div(\Div\A)\frac{1}{n}\Phi(n\xi )\\
    &=  -\Phi^{'}(n\xi )[ \Div(\A \nabla\xi ) + \Div(\Div\A \xi )]  \\ &\quad -  \Phi^{'}(n\xi )[\Div\A\cdot \nabla\xi  
    -\Div(\Div\A \xi )]\\
    &\quad - \Div(\Div\A)\frac{1}{n}\Phi(n\xi )\\
    &= \Phi^{'}(n\xi )[-P^*\xi] + \Phi^{'}(n\xi )
    [\Div(\Div\A ) \xi] -  \Div(\Div\A)\frac{1}{n}\Phi(n\xi )\\
    &=  \Phi^{'}(n\xi ) +   \Div(\Div\A )\left( \Phi^{'}(n\xi )\, \xi - \frac{1}{n}\Phi(n\xi )\right) \\
    &\le 1+ \| (\Div(\Div\A ))_+\|_{\infty}\|\xi\|_\infty.
    \end{align*}
    The last estimate holds because
    \begin{equation*}
        \frac{1}{n\xi}\Phi(n\xi ) = \frac{\Phi(n\xi ) - \Phi(0)}{n\xi} = \Phi'(\zeta)
    \end{equation*}
    for some $\zeta \in (0, n\xi)$ by the Lagrange Mean Value Theorem, whence the convexity of $\Phi$ implies 
    \begin{equation*}
        0 \leq \Phi^{'}(n\xi ) - \frac{1}{n \xi}\Phi(n\xi ) \leq \Phi^{'}(n\xi ) \leq 1.
    \end{equation*}
    It remains to show that $\xi_n\in \mathcal{A}$. By recalling \eqref{eq:P*non-div-P} and applying \eqref{1oszac} and \eqref{eq:divdivterm} we have
    \begin{equation*}
    \begin{split}
         P^\ast \xi_n &= 
         \Phi'(n\xi) P^\ast \xi + n\Phi^{''}(n\xi ) \A \nabla\xi  \cdot \nabla\xi + \Div (\Div \A) (\xi_n - \Phi'(n\xi) \xi)
    \end{split}
    \end{equation*}
    As $P^\ast \xi = 1 \in L^{\infty}(\Omega)$, $0\leq \Phi' \leq 1$, $\Phi \in C^{\infty}(\mathbf{R})$, $\xi \in C^1(\overline{\Omega})$, $\A \in C^2(\overline{\Omega})$ satisfies \eqref{ellipticity}, and $\xi_n \in C^1_c(\overline{\Omega})$, we obtain $P^\ast \xi_n \in L^{\infty}(\Omega)$, hence $\xi_n \in \mathcal{A}$, which ends the proof.
\end{proof}

Consider $$\mathcal{G} = \{u\in L^1(\Omega): u\geq 0,\  -Pu\in L^1_{loc}(\Omega),\ -Pu\geq 0\}.$$
The following lemma shows that $-P(\mathcal{G})\subseteq L^1(\Omega,\delta_\Omega)$.

\begin{lemma}[weighted integrability of elements $F\in P(\mathcal{G})$]\label{lem:FL1loc}
   Assume that $\Omega \in C^{2}$ and that $\A\in C^{2}(\overline{\Omega})$ satisfies \eqref{ellipticity}. If a nonnegative function $u\in L^1(\Omega)$ and a nonnegative measurable function $F\in L^1_{loc}(\Omega)$ satisfy $-Pu = F$ in $\mathcal{D}'(\Omega)$, i.e.
    \begin{align*}  
        -\int_{\Omega} uP^\ast \phi = \int_{\Omega} F\phi,\qquad \phi\in C^\infty_c(\Omega),
    \end{align*}
    then $F\in L^1(\Omega,\delta_\Omega)$.
\end{lemma}

\begin{proof}
    Let $\xi_n \in \mathcal{A}$ be the sequence increasing to $\xi$ given in Lemma~\ref{lem:specialxi}. Lemma~\ref{lem:firstextension} implies that $\xi_n$ can be taken as test functions in the equation. Using the Fatou lemma, the equation for $u$, and the estimate on $-P^*\xi_n$ in Lemma~\ref{lem:specialxi} we get
    \begin{align*}
        \int_\Omega F\xi \leq \liminf\limits_{n\to \infty} \int_\Omega F\xi_n = \liminf\limits_{n\to \infty} -\int_\Omega u P^*\xi_n
        \leq D_{\A}\int_\Omega u 
        =D_{\A}\|u\|_{L^1(\Omega)} <\infty,
    \end{align*}
    which ends the proof, as $\xi \gtrsim \delta_\Omega$ by Lemma~\ref{lem:linear} part \ref{lem:linear_ii}\ref{lem:linear_iic}.
\end{proof}

\subsection{Very weak solutions}

Following \cite{MR2103694} we define
\begin{equation}\label{c1cl}
    C^{1,P}_0(\overline{\Omega}):= \{ \phi\in C^1(\overline{\Omega}): \phi\equiv 0 \ \textnormal{pointwise on}\ \partial\Omega \ {\rm and}\ P^*\phi \in L^\infty (\Omega)\}.   
\end{equation}
The expression $P^*\phi \in L^\infty(\Omega)$ should be understood in the sense of distributions, but under the assumptions of Lemma~\ref{lem:linear} it implies that $\phi \in W^{2,1}(\Omega) \cap C^1(\overline{\Omega}) $, so $P^*\phi$ can be computed almost everywhere, yielding a function in $L^\infty(\Omega)$. Furthermore, the same Lemma implies that $\phi \in C^{1,P}_0(\overline{\Omega})$ satisfies $\phi \lesssim \delta_{\Omega}$. In \cite[Definition~2.3]{MR2103694}, the space $C^{1,P}_0(\overline{\Omega})$ is used as the space of the test functions to define the so-called very weak solutions. 

\begin{definition}[very weak solutions]\label{very-weak}
    We say that $v$ is a very weak solution to 
    \begin{align*}
        \begin{cases}
            -Pu = F\quad &{\rm in}\ \Omega\\
            u\equiv 0\quad &{\rm on}\ \partial \Omega,
        \end{cases}
    \end{align*}
    if \begin{align}\label{eq:veryweak}
        -\int_\Omega u P^*\phi = \int_\Omega F \phi,\quad \phi \in C^{1,P}_0(\overline{\Omega}).
    \end{align}
\end{definition}

\begin{remark}
    The Dirichlet boundary condition is not explicitly specified in the definition, but it is not an omission. The zero boundary values are manifested by the fact that there is no boundary integral in \eqref{eq:veryweak} despite the fact that the test functions are not compactly supported. See \cite[Definition~2.3]{MR2103694} for an analog for nonzero boundary values. 
\end{remark}

The following lemma is one of the key arguments to deduce uniqueness. We believe that it is of independent interest, because of the very weak assumptions on the solution.

\begin{lemma}[distributional solution is a very weak solution]\label{lem:veryweak}
    Assume that $\Omega \in C^{2}$ and that $\A\in C^{2}(\overline{\Omega})$ satisfies \eqref{ellipticity}. Assume that nonnegative $v\in L^1(\Omega)$ satisfies $-Pv = F\in L^1(\Omega,\delta_\Omega)$ in the distributional sense:
    \begin{align}\label{eq:weaklinear}
        -\int_\Omega v P^* \phi = \int_\Omega F\phi,\quad \phi \in C^\infty_c(\Omega),
    \end{align}
    and $v$ satisfies the zero trace condition \eqref{eq:Dirichletcond}. Then $v$ is also a very weak solution to $-Pu = F$.
\end{lemma}

\begin{proof}[Proof of Lemma~\ref{lem:veryweak}]
   Let us fix for every small $\epsilon>0$ a function $\eta_\epsilon\in C_c^\infty(\Omega)$ such that $0\leq\eta_\epsilon\leq 1$, $\eta_\epsilon(x) = 1$ for $\delta_\Omega(x)>2\epsilon$, $\eta_\epsilon(x)=0$ for $\delta_\Omega(x)<\epsilon$, and $\|\nabla \eta_\epsilon\|_\infty \leq C\epsilon^{-1}$ and $\|\nabla^{(2)}\eta_\epsilon\|_\infty \leq C\epsilon^{-2}$ for some fixed  $C>0$ independent of $\epsilon$.

 First, recall that by Lemma~\ref{lem:firstextension} the class of the functions can be extended to 
   \begin{align*}
       \mathcal{A} = \{\phi\in C^1_c(\Omega)\cap W^{2,1}(\Omega): P^*\phi\in L^\infty(\Omega)\}.
   \end{align*} Then, for any $\phi\in C^{1,P}_0(\overline{\Omega})$ and $\epsilon>0$, we have
    \begin{equation} \label{eq:P*cutoff}
    \begin{split}
        P^*(\eta_\epsilon \phi) &= \eta_\epsilon P^* \phi + 2\A \nabla \phi \cdot \nabla \eta_\epsilon + \phi(\Div (\A \nabla \eta_\epsilon) + \Div \A \cdot \nabla \eta_\epsilon) \\
        &= \eta_\epsilon P^* \phi + 2\A \nabla \phi \cdot \nabla \eta_\epsilon + \phi (\A \cdot \nabla^{(2)} \eta_\epsilon + 2 \Div \A \cdot \nabla \eta_\epsilon) 
    \end{split}
    \end{equation} 
    Hence, $\phi\eta_\epsilon\in\mathcal{A}$ and so 
    \begin{align*}
        -\int_\Omega v P^*(\eta_\epsilon \phi) = \int_\Omega F\eta_\epsilon \phi,\quad \phi  \in C_0^{1,P}(\Omega).
    \end{align*}
    Since $\phi(x)\lesssim \delta_\Omega(x)$ (see  Lemma~\ref{lem:linear} part \ref{lem:linear_ii} \ref{lem:linear_iib})  and $F \in L^1(\Omega,\delta_\Omega)$, the right-hand side converges by the Dominated Convergence Theorem to $\int_\Omega F\phi$ as $\epsilon\to 0$.
    
    It remains to show that the left-hand side converges to $-\int_\Omega v P^* \phi$, for which we use \eqref{eq:P*cutoff}. By the Dominated Convergence Theorem and the boundedness of $P^*\phi$ we have 
    \begin{align*}
        \lim\limits_{\epsilon\to 0}\int_\Omega v \eta_\epsilon P^* \phi = \int_\Omega v P^* \phi.
    \end{align*}
    Furthermore, $\phi\in  C^1(\overline{\Omega})$ and $\phi\lesssim \delta_\Omega$, therefore by the properties of $\eta_\epsilon$, 
    \begin{align*}
        &\int_\Omega v |2\A  \nabla \phi \cdot \nabla \eta_\epsilon + \phi(\Div (\A \nabla \eta_\epsilon) + \Div \A \cdot \nabla \eta_\epsilon)|
        \lesssim\frac 1 {2\epsilon} \int_{\{y \in \Omega: \delta_\Omega(y)<2\epsilon\}} v,
    \end{align*}
    which converges to 0 by the Dirichlet condition \eqref{eq:Dirichletcond} for $v$. This ends the proof.
\end{proof}\medskip

\subsection{The Kato-type inequality for very weak solutions} 

The lemma below is an a~priori bound for solutions to linear equations with the right-hand side in the weighted $L^1$ space. We refer to e.g.~\cite{MR2530605} for more results and context on equations of this type.

\begin{lemma}[the estimates in the $L^1$ framework]\label{lem:aprioriL1}
    Assume that $\Omega \in C^{2}$ and that $\A\in C^{2}(\overline{\Omega})$ satisfies \eqref{ellipticity}. Suppose that $F\in L^1(\Omega,\delta_\Omega)$ and $u\in L^1(\Omega)$ is a very weak solution to $-Pu = F$. Then, there exists $C=C(d,\Omega,c_\A,\|\A\|_{C^2(\overline{\Omega})})$ such that
    \begin{align*}
        \|u\|_{L^1(\Omega)} \leq C\|F\|_{L^1(\Omega,\delta_\Omega)}.
    \end{align*}
\end{lemma}

\begin{proof}
      By Lemma~\ref{lem:linear} there exists a unique weak solution $\phi_u\in H^1_0(\Omega)$ to 
    \begin{align*}
        \begin{cases}
            -P^* \phi_u = \sgn u\quad &{\rm in}\ \Omega,\\
            \phi_u \equiv 0\quad &{\rm on}\ \partial \Omega.
        \end{cases}
    \end{align*}
    We also have $\phi_u\in W^{2,p}(\Omega)\cap C^1(\overline{\Omega})$ for all $p\in (1,\infty)$, $-P^\ast \phi_u = \sgn u$ a.e. in $\Omega$ and $\phi_u\equiv 0$ pointwise on $\partial \Omega$.  Lemma~\ref{lem:linear} part \ref{lem:linear_ii} \ref{lem:linear_iib} and the fact that $\|\sgn u\|_\infty \leq 1$ together imply that $\|\phi_u/\delta_\Omega\|_\infty$ is bounded independently of $u$. We also have $\phi_u\in C^{1,P}_0(\overline{\Omega})$, so we can use it as the test function in \eqref{eq:veryweak}, getting
    \begin{align*}
        \int_{\Omega} |u| = -\int_\Omega u P^* \phi_u =\int_\Omega \phi_u F \leq C\int_\Omega |F| \delta_\Omega,
    \end{align*}
    with $C$ independent of $u$, which ends the proof.
\end{proof}

The proof of the next lemma is similar to that given by Oliva and Petitta \cite[Lemma~2.9]{OLIVA} for operators of the form $\Div (\A \nabla u)$.

\begin{lemma}[The Kato-type inequality for very weak solutions]\label{lem:Kato}
    Assume that $\Omega \in C^{2}$ and that $\A\in C^{2}(\overline{\Omega})$ satisfies \eqref{ellipticity}. Assume that $F\in L^1(\Omega,\delta_\Omega)$ and let $u\in L^1(\Omega)$ be a very weak solution to $-Pu = F$. Then,
    \begin{align*}
        -\int_\Omega u_+ P^*\phi \leq \int_{\Omega\cap\{u\geq 0\}} \phi F
    \end{align*}
    holds for all $\phi\in C^{1,P}_0(\overline{\Omega})$ such that $\phi\geq 0$.
\end{lemma}

\begin{proof}
    Let $F_n\in  C^1_c(\Omega)$ approximate $F$ in the $\|\cdot\|_{L^1(\Omega,\delta_\Omega)}$ norm, see e.g.~\cite[Proposition~7.9]{MR1681462}. Then, by Lemma~\ref{lemm:KPR_assumption_verification}, 
    \begin{align*}
        \begin{cases}
            -Pu_n = F_n\quad &{\rm in}\ \Omega,\\
            u_n \equiv 0\quad &{\rm on}\ \partial \Omega,
        \end{cases}
    \end{align*}
    has a unique strong solution $u_n$ such that $u_n\in W^{2,p}(\Omega)\cap W^{1,p}_0(\Omega)$ for all $p\in (1,\infty)$. By the interior regularity theory, see e.g.~\cite[Theorem~9.19]{MR1814364}, and the Sobolev embedding we also have $u_n\in C^2(\Omega)$. Furthermore, by Lemma~\ref{lem:aprioriL1}   we have $\|u_n - u\|_{L^1(\Omega)} \to 0$ as $n\to \infty$. By taking a subsequence we can assume that $u_n$ also converge almost everywhere to $u$. 

    Consider convex functions $\Phi_\epsilon \in C^2(\R)$ that satisfy $\Phi_\epsilon(0) = 0$, $\Phi_\epsilon'(x) = 1$ for $x\geq 0$, $\Phi_\epsilon(x) \geq -\epsilon$ for $x<0$, $0\leq \Phi_\epsilon'\leq 1$, $\Phi_\epsilon'' \leq C\epsilon^{-1}$. Note that $\Phi_\epsilon(u_n)\in W^{2,p}(\Omega)\cap W^{1,p}_0(\Omega)$ for all $p\in (1,\infty)$ and
    \begin{align*}
        P ( \Phi_\epsilon(u_n) ) = \Phi_\epsilon'(u_n) Pu_n + \Phi_\epsilon''(u_n) \A \nabla u_n  \cdot \nabla u_n.
    \end{align*}
    For $\phi \in C^{1,P}_0(\overline{\Omega})$ we have (by Lemma~\ref{lem:linear}) $\phi \in W^{2,p}(\Omega)\cap W^{1,p}_0(\Omega)$ for all $p\in (1,\infty)$. Thus, if $\phi\geq 0$, then using integration by parts, \eqref{ellipticity}, and the convexity of $\Phi$,
    \begin{equation} \label{eq:Kato_approx}
    \begin{split}
        -\int_\Omega \Phi_\epsilon(u_n) P^\ast \phi = -\int_\Omega P \Phi_\epsilon(u_n) \phi &= -\int_\Omega \Phi_\epsilon'(u_n) Pu_n\phi - \int_\Omega \Phi_\epsilon''(u_n) (\A \nabla u_n  \cdot \nabla u_n) \phi\\
        &\leq -\int_\Omega \Phi_\epsilon'(u_n) Pu_n\phi\\
        &= \int_\Omega  \Phi_\epsilon'(u_n) F_n\phi,
    \end{split}
    \end{equation}
    where the last equality holds thanks to $u_n$ being strong solutions. We now examine the convergence of the leftmost and the rightmost expressions in \eqref{eq:Kato_approx}. Since $\Phi_\epsilon'$ is bounded, we have $|\Phi_\epsilon(u_n) - \Phi_\epsilon(u)| \leq C|u_n-u|$, therefore since $\|u_n-u\|_{L^1(\Omega)}\to 0$ and $P^*\phi\in L^\infty(\Omega)$,
    \begin{align*}
        \|\Phi_\epsilon(u_n) P^\ast \phi - \Phi_\epsilon(u) P^\ast \phi\|_{L^1(\Omega)} \leq \|u_n-u\|_{L^1(\Omega)}\|P^* \phi\|_{\infty}\mathop{\longrightarrow}\limits_{n\to \infty} 0.
    \end{align*}
    For the rightmost expression we use the fact that every $\phi \in C^{1,P}_0(\overline{\Omega})$ satisfies $\phi\lesssim \delta_\Omega$, which follows from Lemma~\ref{lem:linear} \ref{lem:linear_iib}:  
    \begin{align*}
        \bigg|\int_\Omega  \Phi_\epsilon'(u_n) F_n\phi - \int_\Omega  \Phi_\epsilon'(u) F\phi\bigg| &\lesssim  \|F_n - F\|_{L^1(\Omega,\delta_\Omega)} +\int_\Omega |\Phi_\epsilon'(u_n)  -\Phi_\epsilon'(u)| |F|\delta_\Omega .
    \end{align*}
    Both expressions tend to 0: the first one because of the assumption on $F_n$ and the second one due to the Dominated Convergence Theorem, because $u_n$ converge to $u$ a.e. and the integrand is majorized by $2|F|\delta_\Omega\in L^1(\Omega)$. 

    Therefore we obtain
    \begin{align*}
        -\int_\Omega \Phi_\epsilon(u) P^\ast \phi \leq \int_\Omega \Phi_\epsilon'(u) F\phi.
    \end{align*}
    We have $|\Phi_\epsilon(u)|\leq \epsilon + u_+  \in L^1(\Omega)$ and $\Phi_\epsilon(u) \to u_+$ as $\epsilon\to 0$. At the same time, $|\Phi_\epsilon'(u)|\leq 1$ and, due to its monotonicity, $\Phi_\epsilon'(u) \to\textbf{1}_{u\geq 0}$ as $\epsilon\to 0$. Finally, $P^\ast \phi \in L^{\infty}(\Omega)$ and $|F\phi|\lesssim |F|\delta_\Omega\in L^1(\Omega)$. Hence, we obtain the statement of the lemma by the Dominated Convergence Theorem. 
\end{proof}

\subsection{Proof of the uniqueness theorem}
We are now ready to prove our main uniqueness theorem.

\begin{proof}[Proof of Theorem~\ref{uniqregu}]
    Let $u,v\in L^1(\Omega)$ be distributional solutions to \eqref{eq:distributional}. By Lemma~\ref{lem:FL1loc} we find that $F_u := f/u^\gamma$, $F_v:= f/v^\gamma$ belong to $L^1(\Omega,\delta_\Omega)$. Furthermore, by Lemma~\ref{lem:veryweak}, $u$ is a very weak solution to $-Pu = F_u$ and $v$ is a very weak solution to $-Pv = F_v$. By subtracting the very weak formulations, we obtain
    \begin{align*}
        -\int_\Omega (u-v)P^* \phi = \int_\Omega \phi(F_u - F_v),\quad \phi\in C^{1,P}_0(\Omega).
    \end{align*}
    If we select a test function satisfying $\phi\geq 0$, then by Lemma~\ref{lem:Kato} we get
    \begin{align*}
        -\int_\Omega (u-v)_+P^* \phi \leq \int_{\Omega\cap\{u\geq v\}} \phi(F_u - F_v)\leq 0,
    \end{align*}
    with the last inequality following from the fact that $F_u\leq F_v$ on $\{u\geq v\}$. In particular, by taking $\phi = \xi$  defined in Section~\ref{sec:weightedintegrability}, we get $(u-v)_+ \equiv 0$. By switching signs we also get $(v-u)_+ \equiv 0$, which means that $u=v$ almost everywhere.
\end{proof}

\section{Final discussion}\label{sec:discussion}

With additional information about $f$ we can deduce better regularity of $u$.

\begin{remark}[better regularity]\label{remarkRegularity}\
    When $f$ has better integrability properties, we obtain better regularity. For example, if 
    \begin{equation} \label{eq:better_f}
        \frac{f}{\delta_\Omega(x)^{\gamma}}\in L^p(\Omega)
    \end{equation}
    for some $1<p<\infty$, we obtain that $u\in W^{2,p}(\Omega)$. This follows from the elliptic regularity theory \cite[Theorem~9.15]{MR1814364} and the fact that $u(x)\gtrsim \delta_\Omega(x)$ for all $x\in \Omega$.
    
    The condition \eqref{eq:better_f} holds for example if we assume that $f \in W^{\lceil \gamma \rceil, p}_0(\Omega)$, where $\lceil \gamma \rceil$ is the least integer larger than or equal to $\gamma$, as follows from the famous characterization of this space via the distance function: it holds for $f \in W^{m, p}(\Omega)$ that
    \begin{equation*}
        f \in W^{m,p}_0(\Omega) \iff \frac{f}{\delta_\Omega(x)^{m}}\in L^p(\Omega),
    \end{equation*}    
    see e.g.~\cite[Theorem~1]{KadlecKufner66} for this result in Lipschitz domains and e.g.~\cite{BalinskyEvans15}, \cite[Theorem V.3.4 and Remark~X.6.8]{EdmundsEvans18}, \cite{KinnunenMartio97}, and \cite{NekvindaTurcinova24} for further extensions and additional context (with the Introduction of \cite{NekvindaTurcinova24} providing a nice modern overview). More precise conditions could also be expressed in terms of the fractional Sobolev spaces, see e.g.~\cite[Theorem~3]{MR4454384} or \cite{Kijaczko}.
\end{remark}

\begin{oq}

Let us finally state some open problems that we believe to be of interest.

\label{general}\rm
\begin{o-problem}\rm
In this paper we have examined the solutions to the PDE of the form
\begin{equation}\label{eq-general}
 -Pu = f(x)\tau (u),
\end{equation}
 where $u$ is positive, $f\in L^1(\Omega)$ is positive and 
$\tau (u)= u^{-\gamma}$ and the operator $P$ is second-order in nondivergence form.
One could ask about existence and uniqueness of solutions to \eqref{eq-general} with the possibly general nonlinearity 
$\tau (\cdot)$ and other operators $P$, for example ones with less regular matrices $\A$, or nonlocal operators with variable coefficients (see \cite{MR3356049} for the case of the fractional Laplacian).
\end{o-problem}

\begin{o-problem}\rm
    In Theorem \ref{maintheorem} \ref{item:gammain12} it was shown that the solution $u$ belongs to $H^1_{loc}(\Omega) \cap W_{0}^{1,q}(\Omega)$ with any $1\le q < \frac{2}{\gamma}$, for $\gamma \in (1,2)$. It would be interesting to know if $u$ belongs to $W_{0}^{1,1}(\Omega)$, when $\gamma =2$ and if one could conclude some more precise regularity when $\gamma >2$. 
\end{o-problem}
\end{oq}

\section*{Acknowledgments}
We thank Tom\'a\v{s} Roskovec for helpful discussions and comments on the manuscript. A.K. and A.R. were supported by the National Science Center (Poland) grant 2023/51/B/ST1/02209.

\bibliographystyle{abbrv}
\bibliography{bib}

\end{document}